\documentclass[12pt]{amsart}
\usepackage[dvipsnames]{xcolor}
\usepackage{fullpage,graphicx,amsmath,amsfonts, amssymb,enumitem, hyperref, nicefrac,esint}
\usepackage[longnamesfirst]{natbib}
\title{Random walk on sphere packings and Delaunay triangulations in arbitrary dimension}
\author{Ahmed Bou-Rabee}
\address{
	Ahmed Bou-Rabee \\
	Department of Mathematics \\
	University of Pennsylvania \\
	David Rittenhouse Laboratory \\
	209 South 33rd Street \\
	Philadelphia, PA 19104 \\
	Email: ahmedmb@gmail.com
}

\author{Ewain Gwynne}
\address{
	Ewain Gwynne \\
	Department of Mathematics \\
	University of Chicago \\
	Eckhart Hall \\
	5734 S University Ave \\
	Chicago IL, 60637 \\
	Email: ewain@uchicago.edu
}

\hypersetup{
	colorlinks=true,
	linkcolor=NavyBlue,
	urlcolor = RoyalBlue,
	citecolor = PineGreen
}

\numberwithin{equation}{section}

\DeclareMathOperator{\R}{\mathbb{R}}

\DeclareMathOperator{\Z}{\mathbb{Z}}
\DeclareMathOperator{\N}{\mathbb{N}}

\DeclareMathOperator{\diam}{{\mathop{diam}}}

\DeclareMathOperator{\conv}{\mathop{\bf conv}}

\newcommand{\eg}{{\it e.g.}}
\newcommand{\ie}{{\it i.e.}}

\newtheorem{theorem}{Theorem}[section]
\newtheorem{prop}[theorem]{Proposition}

\newtheorem{lemma}[theorem]{Lemma}

\newtheorem{remark}[theorem]{Remark}

\newtheorem{problem}[theorem]{Problem}

\newtheorem{theoremA}{Theorem}[section]

\newtheorem{assumptionD}{Assumption}

\renewcommand*{\P}{\mathbb P}

\renewcommand{\div}{\mathop{\bf div}}

\newcommand*{\Zd}{\ensuremath{\mathbb{Z}^d}}
\newcommand*{\Rd}{\ensuremath{\mathbb{R}^d}}

\newcommand{\G}{\mathcal{G}}

\newcommand{\V}{\mathcal{V}}
\newcommand{\E}{\mathcal{E}}
\newcommand{\eps}{\varepsilon}

\newcommand{\CLap}{\boldsymbol{\Delta}}
\newcommand{\Cnabla}{\boldsymbol{\nabla}}

\newcommand{\projone}{\boldsymbol{\mathrm{proj}^{(1)}}}

\newcommand{\projj}{\boldsymbol{\mathrm{proj}^{(j)}}}

\newcommand{\aDelta}{\hyperref[eq:geometric-laplacian]{{\Delta^{\mathcal{G}}_{\cond}}}}
\newcommand{\cond}{\hyperref[eq:geometric-laplacian]{\mathbf{a}}}
\newcommand{\anabla}{\hyperref[eq:anabla]{{\nabla^{\mathcal{G}}_{\cond}}}}
\newcommand{\D}{\hyperref[eq:dir.energy]{\mathcal{D}}}

\newcommand{\GCU}{\hyperref[eq:graph-closure]{\widehat{U}}}
\newcommand{\Qe}{\hyperref[eq:qeset]{Q_e}}
\newcommand{\Lyone}{\hyperref[eq:line-project]{L_y^{(1)}}}
\newcommand{\Lyj}{\hyperref[eq:line-project]{L_y^{(j)}}}

\DeclareMathOperator{\vol}{\mathbf{vol}}
\DeclareMathOperator{\dvol}{\mathbf{d vol}}

\newcommand{\B}{\mathbb{B}}

\DeclareMathOperator{\hm}{\mathop{hm}}
\DeclareMathOperator{\hmn}{\mathop{hm^{n}}}

\DeclareMathOperator{\dPr}{\mathop{d^{Pr}}}

\newcommand{\GVor}{\mathcal{G}^{\mathrm{Vor}}}
\newcommand{\VVor}{\mathcal{V}^{\mathrm{Vor}}}
\newcommand{\EVor}{\mathcal{E}^{\mathrm{Vor}}}

\begin{document}

\begin{abstract}
We prove that random walks on a family of tilings of~$d$-dimensional Euclidean space, with a canonical choice of conductances, converge to Brownian motion modulo time parameterization. 
This class of tilings includes Delaunay triangulations (the dual of Voronoi tesselations)
and sphere packings. 
Our regularity assumptions are deterministic and mild. For example, our results apply to Delaunay triangulations with vertices sampled from a $d$-dimensional Gaussian multiplicative chaos measure. 
As part of our proof, we establish the uniform convergence of certain finite volume schemes for the Laplace equation, with quantitative bounds on the rate of convergence. In the special case of two dimensions, we give a new, short proof of the main result of~\cite{gurel2020dirichlet}.
\end{abstract}
\maketitle 

\vspace{-0.6in}

\tableofcontents

\section{Introduction}
\label{sec:intro}

\begin{figure}
	\includegraphics[width=0.79\textwidth]{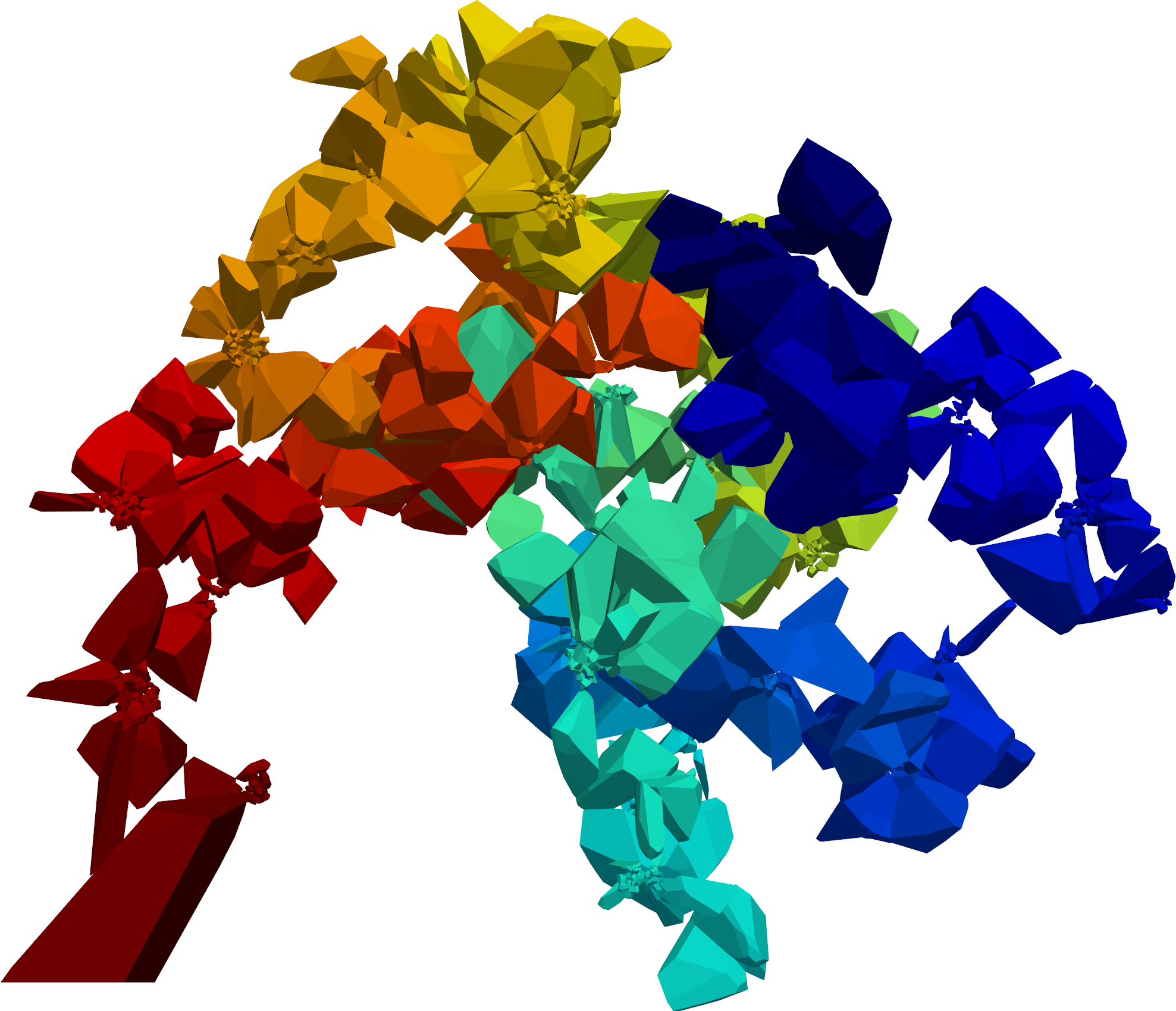}
	\caption{The trace of random walk on the Voronoi tessalation of $2^{19}$ points sampled from a Gaussian multiplicative chaos measure on a cube with parameter~$\gamma =1$. The random walk is stopped upon exiting the cube and the cells of the tessalation are colored according to the order in which they are visited.}
\end{figure}

In this paper, we consider a class of graphs embedded in $\Rd$, called \emph{orthogonal tilings}, which come equipped with a canonical choice of conductances on their edges (defined in Section~\ref{subsec:main-results}). This class includes tangency graphs of sphere packings as well as $d$-dimensional Delaunay triangulations (the dual graphs of Voronoi tesselations). We show that under nearly minimal regularity assumptions, random walks on these graphs converge to Brownian motion modulo time change, equivalently, discrete harmonic functions converge to their continuum counterparts (Theorems~\ref{theorem:convergence-of-rw-trace} and~\ref{theorem:convergence-of-the-dirichlet-problem}). Several motivations drive this work, and these are discussed in detail in Sections~\ref{subsec:prior} and~\ref{subsec:universality}; to summarize:
\begin{itemize}
\item Orthogonal tilings in~$d=2$ have been studied in part to establish universality of statistical mechanics models, see, for example, the ICM survey,~\cite{smirnovicm}. Our work may be viewed as the first step in this program in~$d \geq 3$. 
\item We obtain a higher-dimensional analog of the results of~\cite{gurel2020dirichlet}, which gives the convergence of discrete harmonic functions on planar orthodiagonal lattices to their continuum counterparts  (see also~\cite{MR1692623},~\cite{chelkaksmirnov}, \cite{skopenkov2013boundary}, and~\cite{werness2015discrete} for earlier results). In the special case when $d=2$, we also obtain a short new proof of the main result of~\cite{gurel2020dirichlet}. 
\item Theorem~\ref{theorem:convergence-of-the-dirichlet-problem} can be viewed as a uniform convergence statement for a class of finite volume approximation schemes for the Dirichlet problem. Such approximation schemes are widely used in numerical analysis, see, for example the textbook~\cite{eymard2000finite}. For the class of approximation schemes considered in this paper, convergence was previously only known in $L^2$, which is not sufficient to get convergence 
of the associated random walk. 
\item Theorems~\ref{theorem:convergence-of-rw-trace} and~\ref{theorem:convergence-of-the-dirichlet-problem} apply to highly inhomogeneous tilings, including certain discretizations of higher-dimensional analogs of Liouville quantum gravity (LQG). The statement that random walk on such discretizations converges to Brownian motion modulo time change can be viewed as a higher-dimensional analog of scaling limit results for random walk on discretizations of LQG in two dimensions~\cite{GMS-Tutte, GMSInvariance}, although the setup and proofs in the present paper are quite different.  
\item Our results provide some evidence that sphere packings and Delaunay triangulations are discrete analogs of conformally flat Riemannian metrics, which suggests that such graphs could possibly play a role analogous to planar maps in higher-dimensional random conformal geometry.  
\end{itemize}

\subsection{Main results}
\label{subsec:main-results}

Let~$\mathcal{D}$ be a subset of~$\Rd$, cover~$\mathcal{D}$ by closed, convex polytopes~$\{P_v\}$ with disjoint interiors, and associate each polytope~$P_v$ with a vertex~$v$ in its interior.  Denote the set of such vertices by~$\V$ and declare that $v,w \in \V$ are joined by an edge if $P_v \cap P_w$ is a $(d-1)$-dimensional facet. Assume that each compact subset of~$\mathcal{D}$ intersects finitely many polytopes. We call the resulting graph $\G = (\V,\E)$ a \emph{tiling graph} of~$\mathcal{D}$.  An~\emph{orthogonal} tiling graph satisfies the additional property that for each edge~$(w,v) \in \E$ the vector~$(w-v)$ is orthogonal to the hyperplane containing~$P_w \cap P_v$. 
As we observe in Section~\ref{subsec:natural-examples} below, examples of orthogonal tilings include Voronoi tessellations and tangency graphs of sphere packings.

In this paper we investigate the large scale behavior of random walk on orthogonal tilings. The random walk~$\{X_t\}_{t \geq 0}$ we consider is discrete time and time homogeneous with transition probabilities~$p$ determined by the geometry of the graph: for an edge $(w,v) \in \E$, we set 
\begin{equation} \label{eq:conductance} 
p(w,v) := \frac{\cond(w,v)}{\cond(w)},   \,  \, \,    \text{where}  \,  \, \,   
\cond(w,v) :=  \frac{\vol_{d-1}(P_w \cap P_v)}{\|w - v\| } \,  \, \,  \, \text{and} \,  \, \,  \, \cond(w) := \sum_{(w,v) \in \E} \cond(w, v) .
\end{equation}
Here, $\vol_{d-1}$ denotes $d-1$-dimensional Lebesgue measure. See Figure~\ref{fig:polytope-neighbors}.
\begin{figure}
	\includegraphics[width=0.4\textwidth]{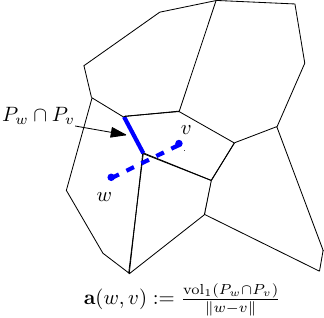}
	\caption{A polytope~$P_v$ and some of its neighbors in an orthogonal tiling. The edge between~$v$ and~$w$ (blue dots) is a blue dotted line and the facet~$P_w \cap P_v$ is a thick blue line. The conductance~$\cond(w,v)$ is the length of the facet divided by the Euclidean distance~$\|w-v\|$.} \label{fig:polytope-neighbors}
\end{figure}

 These probabilities ``tilt" the random walk so that its Euclidean position is a discrete-time martingale, see Lemma~\ref{lemma:linear-lie-in-kernel} and~\cite[Proposition 6.4]{MR4425348}. In fact, simple random walks on orthogonal tilings do not necessarily converge to Brownian motion modulo time change, even under very strong regularity assumptions --- see Theorem~\ref{theorem:counter-example} below.

The main result of this paper gives the convergence of $\{X_t\}_{t\geq 0}$ to Brownian motion modulo time parametrization. 
We recall the definition of the topology on curves modulo time parametrization from~\cite[Section 2]{AB1999}.
Let~$T_1,T_2 > 0$ and let~$\phi_1: [0,T_1] \to \Rd$ and~$\phi_2:[0,T_2] \to \Rd$ be two continuous 
curves. We define 
\begin{equation}
	\label{eq:cmp-topology}
	d(\phi_1, \phi_2) := \inf_{\psi} \sup_{t \in [0,T_1]} | \phi_1(t) - \phi_2(\psi(t))| \, , 
\end{equation}
where the infimum is taken over all increasing homeomorphisms~$\psi:[0,T_1] \to [0,T_2]$.

\begin{theoremA}[Convergence of random walk] \label{theorem:convergence-of-rw-trace}
	Let~$U \subset \Rd$ be a bounded, Lipschitz \footnote{We need to assume the domain is Lipschitz because we need that Brownian motion started near the boundary exits the domain close to its starting point. Without this condition, the random walk could exit the domain much earlier than the Brownian motion.} domain such that~$\overline{U}\subset \mathcal{D}$ and let~$\{\G_n\} = \{(\V_n, \E_n)\}$ 	be a sequence of orthogonal tilings of~$\mathcal{D}$. For~$z \in \mathcal{D}$, let~$z^{n}$ be the nearest
	vertex in~$\G_n$ to~$z$ with ties broken in lexicographical ordering. Assume that  
	\begin{equation}
		\label{eq:smallness}
	 \lim_{n \to \infty} \left( \eps_n + \sup_{z \in \mathcal{D}} \|z^{n} - z\| \right) = 0  , \quad \text{where} \quad
	 \eps_n := \sup_{v \in \V_n} \diam(P_v)  
	\end{equation}
	and at least one of the following three conditions is satisfied:
	\begin{enumerate}[label = (\Roman*)]
		\item planarity,~$d = 2$; or \label{item:planarity}
		\item the volume of the smallest tile in~$\G_n$ is at least~$exp( -o(\eps_n^{-1}))$; or \label{item:vol}
		\item there exists $\alpha> 0$ such that the diameter of each tile~$P_v$ in~$\G_n$ is at most \\ $O(\max_{(v,w) \in \E_n} \|w-v\|^{\alpha} \vol_{d-1}(P_w \cap P_v)^{\alpha})$, where the max is over all edges incident to $v$. 
		  \label{item:diam}
	\end{enumerate}
	Then, for each~$z \in U$, as~$n \to \infty$, the linearly interpolated random walk~$\{X_t\}$
	started at~$z^{n}$ and stopped upon exiting~$U$ 
	converges in law to standard Brownian motion started at~$z$ and stopped upon exiting~$U$ with respect to the metric on curves viewed modulo time parametrization~\eqref{eq:cmp-topology}. Moreover, the convergence is uniform over all choices of $z$.
\end{theoremA}

The hypothesis~\eqref{eq:smallness} is close to necessary for the conclusion of Theorem~\ref{theorem:convergence-of-rw-trace} to hold. 
The hypotheses~\ref{item:vol} and~\ref{item:diam} are quite mild and are true for essentially any models one might be interested in, even discretizations of rough, fractal geometries (see Proposition~\ref{prop:min-max-edge-length}). However, we do not know whether it is necessary to have one of these two hypotheses for $d\geq 3$ (see Problem~\ref{problem:hypotheses}).

Hypothesis~\ref{item:vol} can be replaced by a slightly weaker hypothesis concerning the diameters of the connected components of the set of small tiles of $\G_n$, see Theorem~\ref{theorem:generaldim}. This weakened hypothesis is similar to~\cite[Assumption 1.2]{CLM}. We emphasize that in Hypothesis~\ref{item:diam}, we allow any $\alpha > 0$. For $\alpha \not=1/d$, the ratio $\|w-v\|^{\alpha} \vol_{d-1}(P_w \cap P_v)^{\alpha} / \mathrm{diam}(P_v)$ is not scale invariant, so the bound we require is much weaker than a uniform control over the geometry of the tiling at all scales. 
 \smallskip

The macroscopic features of a discrete time Markov process are closely tied to the large scale behavior of its generator, see, \eg, \cite[Theorem 19.25]{MR4226142}. The generator of~$X_t$ is
\begin{equation}
	\label{eq:geometric-laplacian}
	\aDelta h(v) = \sum_{(w,v) \in \E}  \cond(w,v)( h(w) - h(v))
	\quad \mbox{where}  \quad 	\cond(w,v) :=  \frac{\vol_{d-1}(P_w \cap P_v)}{\|w - v\| } , \, 
\end{equation}
as in~\eqref{eq:conductance}.
Theorem~\ref{theorem:convergence-of-rw-trace} is roughly equivalent to the statement that~$\aDelta$-harmonic functions are approximated uniformly, at large scales, by continuum harmonic functions. 

In what follows, we write~$\V_n[U]$ and~$\partial \V_n[U]$ for the set of vertices of $\G_n$ in the interior of~$U$ and those which share an edge with a vertex in the interior, respectively. 
\begin{theoremA}[Convergence of the Dirichlet problem] \label{theorem:convergence-of-the-dirichlet-problem}
	Let~$U \subset \mathcal{D}$ be a bounded domain and let~$h_C$ be a function which is harmonic on~$U$. Assume that at least one of the following 
	two conditions hold:  
	\begin{enumerate}[label = (\alph*)]
		\item $h_C$ is harmonic in a neighborhood of~$\overline{U}$; or \label{item:cont-neighborhood}
		\item $h_C$ is continuous on~$\overline{U}$ and $U$ is Lipschitz and~$\overline{U} \subset \mathcal{D}$. \label{item:hypb}
	\end{enumerate}
	 If~$\G_n$ satisfies~\eqref{eq:smallness} and one of \ref{item:planarity}, \ref{item:vol}, or \ref{item:diam} in Theorem~\ref{theorem:convergence-of-rw-trace}, 
	then the discrete harmonic extension~$h^{n}_D$ of~$h_C$ to~$\V_n[U]$, defined by, \footnote{In the case~$h_C$ is not defined at a boundary vertex~$v$, pick the nearest point in the cell~$P_v$
		for which it is defined.}
	\[	
	\left\{	
	\begin{aligned}
		&  \aDelta h^{n}_D = 0 & \mbox{in} & \ \V_n[U] \\
		& h^{n}_D = h_C  & \mbox{on} &  \ \partial \V_n[U]    \, ,
	\end{aligned}
	\right.
	\]
	converges uniformly to~$h_C$,

	\begin{equation}
		\label{eq:convergence-of-d-general}
	\lim_{n \to \infty} \sup_{z \in \V_n[U]} |h^{n}_D(z) - h_C(z)| = 0  \, .
	\end{equation}
\end{theoremA}	
\noindent In fact, we obtain a quantitative bound on the rate of convergence under Assumption~\ref{item:cont-neighborhood} in~\eqref{eq:convergence-of-d-general}, see Theorems~\ref{theorem:2d},~\ref{theorem:generaldim}, and~\ref{theorem:generaldim-decay} below.

\subsection{Examples}
\label{subsec:natural-examples}

\subsubsection{Sphere packings}

A \emph{sphere packing} is a collection of~$d$-dimensional spheres in~$\R^d$ with disjoint interiors. The spheres are not required to have the same radii. A \emph{sphere-packed graph}, $\G = (\V,\E)$, is the tangency graph of a sphere packing. Each vertex~$v \in \V$ is associated to a sphere~$S_v$ with~$v$ as its centerpoint, and for each edge~$(v, w) \in \E$ there is a unique 
hyperplane~$R_{v,w}$ which is tangent to both~$S_v$ and~$S_w$ at the unique point of $S_v\cap S_w$. A vertex~$v$ is \emph{covered} if the intersection of the hyperplanes~$\{R_{v,w}\}_{(v,w) \in \E}$ is a bounded $d$-dimensional polytope $P_v$ containing $v$. A sphere-packed graph is~\emph{covering} if every vertex is covered. A sphere packing~\emph{covers} a domain~$U$ 
if the spheres in the packing cover~$U$ and every vertex whose corresponding sphere intersects~$U$ is covered. 
It is immediate from the definitions that a locally finite sphere packed graph which covers a domain is an orthogonal tiling of that domain.

\subsubsection{Voronoi tesselations and Delaunay triangulations}

The~\emph{Voronoi tesselation} of a locally finite set of points~$S \subset \Rd$, is the partition of~$\Rd$
into cells~$\{\mathcal{C}_s\}_{s \in S}$ such that every point in the tile~$\mathcal{C}_s$ is closer
to~$s$ than any other point~$s' \in S$, 
\[
\mathcal{C}_s = \{y \in \Rd: \|s - y\| \leq \|s' - y\| \quad \forall s' \in S  \} \, . 
\]
Equivalently, see, \eg ,\cite[Section 1]{mollerLectures}, the cell~$\mathcal{C}_s$ is the intersection, 
over all~$s'\in S$ with~$s' \neq s$ of the closed halfspace 
containing~$s$ and bounded by the bisecting hyperplane of~$s$ and~$s'$, 
\begin{equation}
	\label{eq:half-space-definition-of-cells}
	\mathcal{C}_s = \bigcap_{s' \in S : s' \neq s} \{ y \in \Rd:  (y - \frac{1}{2}(s + s')) \cdot (s-s') \geq 0 \} \, . 
\end{equation}
Consequently, cells are closed~$d$-dimensional convex sets; and for any two neighbors~$s,s'$, the
vector~$(s-s')$ is orthogonal to the~$(d-1)$-dimensional facet~$\mathcal{C}_s \cap \mathcal{C}_{s'}$.
If the cells are finite we denote the orthogonal tiling induced by the Voronoi tesselation of a point set~$S \subset \Rd$ by~
\begin{equation}
	\label{eq:voronoi-tiling-graph}
	\GVor(S) = (\VVor(S), \EVor(S)) \, .
\end{equation}
This is the~\emph{Delaunay triangulation} of the Voronoi tesselation.

Let $U\subset \Rd$ be a bounded domain. Let~$\mu$ be a Radon measure defined in a neighborhood of $\overline U$
which satisfies, for some exponents~$\beta^{\pm} > 0$, for all sufficiently small~$r \in (0,1)$, the bound 
\begin{equation}
	\label{eq:holder-assumption}
	r^{\beta^{+}} \leq \mu(B_r(z)) \leq r^{\beta^{-}} \,  \ \quad \forall z \in U , 
\end{equation}
	where~$B_r(s)$ is the Euclidean ball of radius~$r$ centered at~$s \in \Rd$. 
The following proposition implies that Hypothesis~\ref{item:vol} is satisfied  by~$\GVor(\Lambda_m)$,
where~$\Lambda_m$ is a Poisson point process of intensity~$m \mu$, for large~$m$ almost surely.
\begin{prop}
	\label{prop:min-max-edge-length}
	Let~$U$ be a bounded domain.
	With probability one, for all sufficiently large~$m$, the cells of the Voronoi tesselation of~$\Lambda_m$
	which intersect~$U$ satisfy, 
	\begin{equation}
		\label{eq:cells-are-contained-in-balls}
	B_{m^{-8/\beta^-}}(s) \subset \mathcal{C}_s \subset B_{m^{-1/(3\beta^+)}}(s) \qquad \forall s \in \Lambda_m  \mbox{ with~$\mathcal{C}_s \cap U \neq \emptyset$} \, .
	\end{equation}
	In particular,the conclusions of Theorem~\ref{theorem:convergence-of-the-dirichlet-problem} and, if~$U$ is Lipschitz, Theorem~\ref{theorem:convergence-of-rw-trace}, hold for the graphs $\GVor(\Lambda_m)$ as $m\to\infty$.
\end{prop}

\begin{remark} \label{remark:gmc}
	Let $\Phi$ be a log-correlated Gaussian field on a bounded domain $U \subset \Rd$ and let $\gamma \in (0,\sqrt{2d})$. Let $\mu = ``e^{\gamma \Phi(x)} \,dx_1\dots dx_d"$ be the associated Gaussian multiplicative chaos (GMC) measure. It follows from standard estimates for GMC that there exists $\beta^\pm = \beta^\pm(d,\gamma)> 0$ such that a.s.~\eqref{eq:holder-assumption} holds for $\mu$ for all sufficiently small $r\in (0,1)$ (how small is random). For example, this can be deduced from~\cite[Theorem 3.26]{berestycki2024gaussian} and a union bound argument. Hence, Proposition~\ref{prop:min-max-edge-length} applies to GMC measures on domains in $\Rd$.
\end{remark}

\subsubsection{Orthodiagonal maps}
\label{subsection:orthodiagonal}
\begin{figure}
\includegraphics[width=0.3\textwidth]{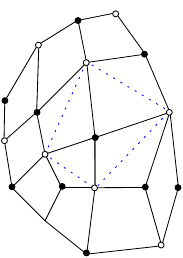}
\caption{An orthodiagonal map. The vertices in~$\V^{\bullet}$ ($\V^{\circ}$) are black (white) and the edges~$\E$ are black lines.
The polytope corresponding to the center vertex is outlined in dotted blue.
 } 
\label{fig:orthogonal-diagonal}
\end{figure}

An \emph{orthodiagonal} map is a finite connected planar graph in which each edge is a straight line segment, each inner face is a quadrilateral with orthogonal diagonals; and the boundary of the outer face is a simple closed curve. An orthodiagonal map 
can be expressed as~$\G = ([\V^{\bullet}, \V^{\circ}], \E)$, where~$\V^{\bullet}, \V^{\circ}$ is a bipartition of the vertices of the graph. 
The vertices~$\V^{\bullet}$ induce a graph~$\G^{\bullet} = (\V^{\bullet}, \E^{\bullet})$ where there is an edge between two vertices~$v_1, v_2 \in \V^{\bullet}$, if they lie on the same inner face of~$\G$. The vertices~$\V^{\circ}$ similarly induce a graph~$\G^{\circ} = (\V^{\circ}, \E^{\circ})$. Isoradial graphs, discussed at the beginning of Section~\ref{subsec:universality} below, can be represented as an orthodiagonal map, see~\cite[Section 2.2]{gurel2020dirichlet}.

The conductance of an edge~$(v_1, v_2)$ corresponding to the inner face of a quadrilateral $(v_1, w_1, v_2, w_2)$ is given by 
\[
\mathbf{a}(v_1, v_2) := \frac{ \|w_1 - w_2\|}{\|v_1 -v_2\|} \, . 
\]
It was shown in~\cite{gurel2020dirichlet} that functions~$f: \V^{\bullet} \to \R$ which are~$\mathbf{a}$-discrete harmonic, that is, which satisfy,
\[
\sum_{(v, w) \in \E^{\bullet}} \mathbf{a}(v, w) (f(v) - f(w)) = 0  \, , 
\]
are close to continuous harmonic functions (in the sense of Theorem~\ref{theorem:convergence-of-the-dirichlet-problem}). 

When every quadrilateral in~$\G$ is convex,~$\G^{\bullet}$ can be represented as an orthogonal tiling. Each vertex~$v \in \V^{\bullet}$
which is not on the outer face of~$\G$ has a polytope~$P_v$ with edges given by the opposite diagonals of the quadrilaterals
containing~$v$ --- that is for each internal face~$(v, w, v', w')$ of~$\G$ with~$v' \in V^{\bullet}$ and~$w,w' \in V^{\circ}$  the edge~$(w, w')$ --- see Figure~\ref{fig:orthogonal-diagonal}. It is straightforward to see for each~$v, v' \in \V^{\bullet}$ that~$\mathbf{a}(v, v')$ coincides
with the conductance defined in~\eqref{eq:geometric-laplacian} and that convexity of each quadrilateral ensures~$v$ is contained in the interior of~$P_v$ and that~$P_v$ is a simple, convex polygon.


We describe how to adapt our proof to the case where non-convex quadrilaterals are allowed in Remark~\ref{remark:orthodiagonalmaps} below.

\subsection{Prior work and proof idea}
\label{subsec:prior}

The equation~\eqref{eq:geometric-laplacian} is a~\emph{finite volume scheme} on~$\G_n$ and it is well known that the convergence in Theorem~\ref{theorem:convergence-of-the-dirichlet-problem} holds (with no additional assumptions on~$\G_n$) in~$L^2$, see, \eg, \cite[Theorem 9.3]{eymard2000finite}. Briefly, a finite volume scheme is a numerical method for solving divergence-form elliptic partial differential equations. The domain is partitioned into ``control volumes'' and integrating the equation by parts
on each volume leads to a finite set of constraints. In the case of the Laplacian on orthogonal tilings this gives the operator~$\aDelta$. 

 Our main contribution can be thought of as an~$L^2$ to $L^{\infty}$ estimate for the operator~$\aDelta$ under very weak regularity assumptions on the graph.  As discussed above, these weak assumptions allow the result to apply to highly inhomogeneous random graphs arising from approximations of random fractal geometries. 

$L^2$ to $L^\infty$ estimates are established in previous work when~$\aDelta$ is \emph{uniformly elliptic}, that is, when~$\cond$ is bounded uniformly from above and below, see for example~\cite[Proposition 5.3]{delmotte} for an interior~$L^2$ to $L^{\infty}$ estimate for discrete harmonic functions. The main difficulty in proving Theorem~\ref{theorem:convergence-of-the-dirichlet-problem} is that, in general, the operator~$\aDelta$ is not uniformly elliptic~\emph{at any scale}. 

We prove~\eqref{eq:convergence-of-d-general} by an iteration of an~$L^2$ bound (Proposition~\ref{prop:energy-bound}) and a Poincar\'e type inequality (Lemma~\ref{lemma:ptype-inequality}). The proof, although distinctly different, is reminiscent of the Campanato-type large-scale regularity iteration originating in~\cite{armstrongsmart} and exposited in~\cite[Chapter 4]{AKMBook}. The idea there, which has become a cornerstone of the theory of elliptic homogenization, is that if the solution 
of a heterogeneous equation can be approximated by a harmonic function across a wide range of scales, it should be as regular as harmonic functions.

Here our ``harmonic approximation'', Proposition~\ref{prop:energy-bound}, has an error 
that depends on the continuum harmonic function all the way up to the boundary. Since our graphs are possibly very irregular, there is no immediate way to ``mollify away'' this error. This prevents us from applying the iteration in~\cite{armstrongsmart}. 

Instead, we fix the continuum harmonic function and iterate the discrete approximation. We first show, by the~$L^2$ estimate and a Poincar\'e-type inequality, that the discrete harmonic function,~$h_D$, is pointwise close to the continuum one,~$h_C$, on a dense set of columns (Proposition~\ref{prop:energy-poincare-bound}). The remaining ``bad'' set where~$|h_D - h_C|$ is large has small measure. By Proposition~\ref{prop:energy-poincare-bound} again, the discrete harmonic extension of~$h_C$ to the bad set is close to~$h_C$ on a set of even larger measure. Repeating this, using one of assumptions \ref{item:vol} or \ref{item:diam} in Theorem~\ref{theorem:convergence-of-rw-trace} to close the loop, yields the pointwise bound. 

When~$d=2$, the case which has received the most attention in previous literature, closeness on a dense set of columns and the maximum principle immediately imply the pointwise bound everywhere, yielding a short proof of Theorem~\ref{theorem:convergence-of-the-dirichlet-problem}. This two dimensional result slightly improves~\cite[Theorem 1.1]{gurel2020dirichlet} in that we do not require the domains to be simply connected.

Earlier work in dimension two, for instance~\cite{MR1692623},~\cite{chelkaksmirnov}, \\ \cite{skopenkov2013boundary}, and~\cite{werness2015discrete}, imposed various assumptions on the graph which ensure that the operator~$\aDelta$ is uniformly elliptic. A similar result for finite volume schemes of a convection diffusion equation in dimensions two and three was also proved (again under assumptions implying uniform ellipticity) in~\cite[Corollary 1]{MR1863279}.

\subsection{Universality and random geometry} 
\label{subsec:universality}
A special case of orthogonal tilings in two dimensions are \emph{isoradial graphs}, graphs which have an embedding into~$\R^2$ so that every face~$F$ of the graph lies on a circle of radius 1 with center in the interior of~$F$.  The term was coined in~\cite{kenyon2002} and a theory of discrete analytic functions on these graphs was developed in~\cite{duffin} and later, independently in~\cite{mercat}. 

Convergence of the Dirichlet problem on isoradial graphs,~\cite{chelkaksmirnov}, was used in a companion paper,~\cite{MR2957303}, to prove
convergence of the critical Ising model, thus establishing~\emph{universality} for the model. 
In an ICM survey~\cite{smirnovicm} asked if these results generalize to other two-dimensional graphs and higher dimensions; in particular, our work addresses~\cite[Questions 1 and 12]{smirnovicm}. This question was also highlighted again in~\cite[Problem 5.9]{skopenkov2013boundary}. 
Several other models have been studied on isoradial graphs including critical dimers, loop-erased random walk, spanning trees, the random-cluster model, and bond percolation; see the survey~\cite{MR3372861} and the papers~\cite{MR3252428, MR3201923,MR3621833, MR3858924}. An interesting next step would be to consider these models for~$d \geq 3$
on orthogonal tilings, see Problem~\ref{problem:lerw} below.

This paper can also be seen as a contribution towards understanding higher-dimensional analogues of discrete conformal geometry and~\emph{Liouville quantum gravity} (LQG). Roughly speaking, LQG is a theory of random geometry in two dimensions which describes the scaling limit of discrete random surfaces, such as uniform triangulations of the sphere, see the surveys~\cite{GwynneSurvey, MR4680280} and the book~\cite{berestycki2024gaussian}. 

There has been recent interest in establishing a corresponding theory in~$d \geq 3$, see, for example,~\cite{schiavo2021conformally, cercle2022, ding2023tightness}. These works study a notion of continuum random geometry in higher dimensions, which is described by a random Riemannian metric tensor of the form $g = e^{\gamma \Phi} g_0$, where $\gamma \in (0,\sqrt{2d})$ is a parameter, $\Phi$ is a log-correlated Gaussian field on a $d$-dimensional manifold (or a minor variant thereof), and $g_0$ is a fixed smooth background metric. If we take $g_0$ to be the flat metric, then (at least at a heuristic level) the random metric tensor $g$ is \emph{conformally flat}, meaning that it is a scalar function times the flat metric. It remains open to determine the correct discrete analog of this continuum theory (analogous to random planar maps in~$d=2$). See, \eg,~\cite[Section 1.2]{ding2023tightness}. The results of this paper contribute to this problem in two respects.

A natural approach to finding a discrete analog of LQG in $d$ dimensions is to search for a class of graphs embedded in $\mathbb R^d$ which satisfy a discrete analog of ``conformal flatness". In~\cite{curienbenjamini} sphere packings were studied partly to understand higher dimensional quantum gravity and it was asked,~\cite[Question 2]{curienbenjamini}, if being representable as the tangency graph of a sphere packing is a discrete analog of being ``conformally flat''. Theorem~\ref{theorem:convergence-of-rw-trace} may be viewed as positive evidence for this, since it implies that the scaling limit of random walk on sphere packings is Brownian motion modulo time change, \ie, in the scaling limit there is no directional bias depending on where the walk is in space.

 Gaussian multiplicative chaos (GMC),~\cite{kahane1985chaos,rhodes2014gaussian,berestycki2024gaussian}, provides the analogue of the LQG area measure in~$d \geq 3$. Voronoi tessellations with centers sampled according to a Poisson point process with respect to GMC (trivially) converge to GMC in the large sample size limit. Thus, Proposition~\ref{prop:min-max-edge-length} and Remark~\ref{remark:gmc} give the convergence to Brownian motion (modulo time change) for random walk on certain embedded random graphs which approximate higher-dimensional analogs of LQG. The first papers to establish this sort of result in two dimensions were~\cite{ GMS-Tutte, GMSInvariance} which prove the convergence of random walk on \emph{mated-CRT maps}, a family of random planar maps known to converge to Liouville quantum gravity. The setting and proof of this result, however, are quite different from the present paper. The proof uses both planarity and the randomness of the environment (whereas here we have a fixed, deterministic environment in arbitrary dimension). Moreover, there are no conductances on the edges of the mated-CRT map, so the random walk is not a martingale (although the proofs show that it is approximately a martingale at large scales). See also Problem~\ref{problem:higherdim-invariance}.

\subsection{Open questions}
We mention some possible directions for future research motivated by our work. 
\begin{problem} \label{problem:lerw}
	What can be said about other statistical mechanics models, \eg, loop-erased random walk, uniform spanning tree, percolation, Ising model, dimers, discrete Gaussian free field, etc., on orthogonal tilings for $d\geq 3$?
\end{problem}

The paper~\cite{GMSInvariance} proves that random walk converges to Brownian motion modulo time change in certain random planar environments which are not stationary with respect to translations, but are instead only ``translation invariant modulo scaling". It is plausible that this result could be extended to dimension $d\geq 3$ if one assumes that the edge conductances scale by $\Lambda^{d-2}$ when one scales space by $\Lambda$, see~\cite[Section 1.5]{GMSInvariance}. It appears that the main obstacle to proving such an extension is transferring from an $L^2$ bound for discrete harmonic functions to an $L^\infty$ bound (this corresponds to Lemma 2.19 in \cite{GMSInvariance}, the proof of which strongly uses planarity).

\begin{problem} \label{problem:higherdim-invariance}
Can the techniques of this paper be used to prove a higher dimensional analogue of \cite{GMSInvariance}?
\end{problem}

A consequence of Theorem~\ref{theorem:convergence-of-rw-trace} is that random walk on Voronoi tessellations with centers sampled from a $d$-dimensional Gaussian multiplicative chaos measure converges to Brownian motion modulo time change. 
The paper~\cite{GB-LBM} proves that random walk on mated-CRT maps converges in law with respect to the uniform topology to \emph{Liouville Brownian motion}, the natural LQG time change of Brownian motion. It would be interesting to establish an analog of this result for the Voronoi tessellations considered in this paper.  

\begin{problem} \label{problem:convergence-to-lqg}
Determine the scaling limit of random walk on Voronoi tessellations with centers sampled from $d$-dimensional GMC viewed as a parametrized path which spends one unit of time in each cell.
\end{problem}

Instead of giving the walk its natural parametrization, we could alternatively seek a time parametrization which makes it converge to standard Brownian motion. To be consistent with Brownian scaling, we need to parametrize so that if we scale space by $\Lambda$, the amount of time spent in each polytope scales like $\Lambda^{2 }$. 

\begin{problem}
	Consider random walk on an orthogonal tiling graph parameterized so that the amount of time
	it spends at each polytope is proportional to the volume of the polytope to the power $2/d$.
	Does this random walk converge to standard Brownian motion? 
\end{problem}

\begin{problem} \label{problem:hypotheses}
For $d\geq 3$, either show that the conclusion Theorem~\ref{theorem:convergence-of-rw-trace} is true without assumption \ref{item:vol} or \ref{item:diam}; or give an explicit counterexample to show that at least one of these assumptions is necessary. 
\end{problem}

\subsection{Paper outline} 
We start in Section~\ref{sec:prelims} with discrete PDE estimates, geometric bounds for orthogonal tilings, and the aforementioned~$L^2$ estimate for~$\aDelta$-harmonic functions. All of the results contained in the section are standard but included for the reader's convenience. In Section~\ref{sec:iteration} we prove Theorem~\ref{theorem:convergence-of-the-dirichlet-problem} under Hypothesis~\ref{item:cont-neighborhood} and at least one of Assumptions~\ref{item:planarity},~\ref{item:vol}, or~\ref{item:diam}. 
Then, in Section~\ref{sec:rw-convergence} we use this to prove Theorem~\ref{theorem:convergence-of-rw-trace} and subsequently deduce Theorem~\ref{theorem:convergence-of-the-dirichlet-problem} under Hypothesis~\ref{item:hypb}. In fact, there we prove a ``black box'' statement asserting, roughly, that whenever one has convergence of the Dirichlet problem, one has convergence of the associated random walk. In Section~\ref{sec:verifying-assumptions} we prove Proposition~\ref{prop:min-max-edge-length}, which verifies that our results apply to graphs derived from Gaussian multiplicative chaos measures. 
Finally, in Section~\ref{sec:counter-example} we give an example of a sphere packing in~$d \geq 3$ for which simple random walk (with unit conductance)
does not converge to Brownian motion modulo time parameterization.

\subsection*{Conventions and notation}
\begin{itemize}
	\item For a set~$U \subset \Rd$, we denote the~$d$-dimensional Lebesgue measure by $\vol_{d}(U)$
	and we denote its Euclidean diameter by~$\diam(U)$.
	\item For a point~$v \in \Rd$,~$\|v\|$ is the Euclidean norm.
	\item For a point~$v \in \Rd$ and~$r > 0$, we write~$B_r(v)$ for the Euclidean ball of radius~$r$ centered at~$v$. When~$v=0$, we write~$B_r$. 
	\item Unless explicitly mentioned, we henceforth fix a set~$\mathcal{D}$ and an orthogonal tiling~$\G$ of~$\mathcal{D}$ and denote the set of edges and vertices by~$\E$ and~$\V$ respectively. 
	We also denote, for~$e = (w,v) \in \E$, 
	\[
	H_e := P_w \cap P_v \, . 
	\]
	\item In an abuse of notation, we identify vertices~$v \in \V$
	with their location in~$\Rd$ and the edge~$(w,v) \in \E$ with 
	the straight line between~$w$ and~$v$. We also denote the length of the edge~$e = (w,v) \in \E$, by~$\|e\| := \|w - v\|$.

	\item 
	The induced subgraph of a domain~$U \subset \Rd$ is~$\G[U]$. It is defined as the set of vertices contained in the interior of~$U$ 
	together with the edges connecting pairs of vertices in that subset --- see Figure~\ref{fig:induced-subgraph}. The edges and vertices of an induced subgraph are~$\E[U]$ and~$\V[U]$ respectively. 
	
	\item The set of edges connecting pairs of vertices in a set~$A$ is denoted by~$\E[A]$. 
	The set of edges in~$\E$ with at least one vertex in~$A$ is~$\E[\overline{A}]$ and~$\partial \E[A] = \E[\overline{A}] \setminus \E[A]$.
	In an abuse of notation, we consider~$\E$ to be a set of both directed and undirected edges.

	\item 
	The sets~$\partial U$,~$U^{\circ}$, and~$\overline U$ denote the boundary, interior, and closure of~$U$, a subset of~$\Rd$. 
	\item For a collection of vertices~$A \subset \Rd$, we write~$\partial A$ for the set of vertices in~$\V \setminus A$ which share an edge with a vertex in~$A$
	and~$\overline A = A \cup \partial A$.

	\item  
	For a domain~$U \subset \Rd$ we define the set
	\begin{equation}
		\label{eq:graph-closure}
		\GCU :=  \bigcup_{v \in \overline{\V[U]}} P_v  \, , 
	\end{equation}
	as the set of polytopes with vertices in or adjacent to~$U$.

	\item A \emph{vector field}~$\theta$ is a function~$\theta: \E[A] \to \R$ which is antisymmetric:
	~$\theta(w,v) = - \theta(v,w)$ for every~$(w,v) \in \E[A]$.

\end{itemize}

\begin{figure}
	\includegraphics[width=0.5\textwidth]{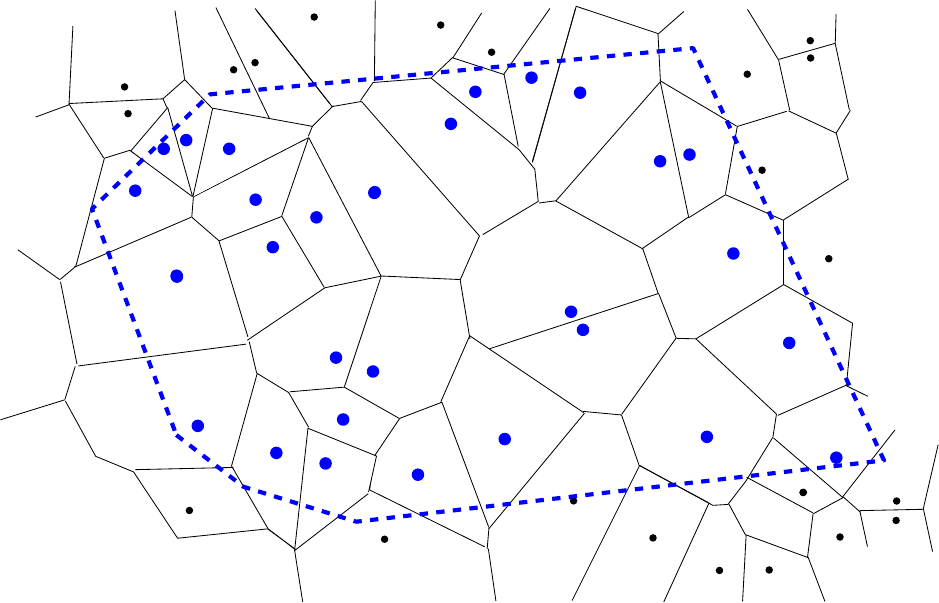}
	\caption{The domain~$U$ is outlined in dashed blue and the vertices of the induced subgraph~$\G[U]$ are blue dots.}
	\label{fig:induced-subgraph}
\end{figure}

\subsection*{Acknowledgments}
Thank you to Gioacchino Antonelli, Scott Armstrong, Bill Cooperman, and Marianna Russkikh for useful discussions and suggestions.
Thank you to Leonardo Bonanno, Jeanne Boursier, Martina Neuman, and the anonymous referee for helpful comments on an earlier
version of this article. 
A.B. was partially supported by NSF grant DMS-2202940. E.G. was partially supported by NSF grant DMS-2245832.

\section{Preliminaries}
\label{sec:prelims}
Here we collect, for completeness, some standard preliminary results that are essentially contained in, \eg, the textbooks~\cite{eymard2000finite, Telcs2006, droniou2018gradient}.

\subsection{Discrete PDE}
We write~$\CLap$ and~$\Cnabla$ for the continuum Laplacian and gradient on~$\Rd$. For a function~$f : \V \to \R$ the vector fields
~$\nabla^{\mathcal{G}} f$ and~$\anabla f$ are defined by,  for a directed edge~$e = (w,v)$, 
\begin{equation}
	\nabla^{\mathcal{G}} f(e) = f(w) - f(v)	
\end{equation}
and
\begin{equation}
	\label{eq:anabla}
\anabla f(e) = \cond(e)(f(w) - f(v)) \, . 
\end{equation}
 The~\emph{divergence} of a vector field~$\theta:\E \to \R$ is the scalar function 
\begin{equation}
	\label{eq:discrete-divergence}
(\div \theta)(v) := \sum_{ (w,v) \in \E } \theta(e) \, , 
\end{equation}
where the sum is over all the directed edges whose endpoint is~$v$. Observe that we can write 
\[
\aDelta h(v) = (\div \anabla h)(v) \, .  
\]

We first recall a version of the discrete divergence theorem. This holds for general graphs
and is well known, but we include a (short) proof. 
\begin{lemma} \label{lemma:discrete-divergence-theorem}
		Let~$A \subset \V$ be a finite set of vertices, and let~$f: \overline{A} \to \R$ be a function which is
	zero on~$\partial A$. Then, for any vector field~$\theta: \E[\overline{A}] \to \R$ we have that 
	\begin{equation}
		\label{eq:discrete-divergence-theorem}
	\sum_{v \in A} (\div \theta)(v) f(v) = - \frac{1}{2} \sum_{e \in \E[\overline{A}]} \theta(e) \nabla^{\mathcal{G}} f(e) \, . 
\end{equation}
\end{lemma}
\begin{proof}
	Each neighboring pair of vertices~$\{w,v\}$ appears twice in the sum on the right in~\eqref{eq:discrete-divergence-theorem}, once as~$\theta(w,v) (f(w) - f(v))$
	and again as~$\theta(v,w) (f(v)-f(w)) = \theta(w,v) (f(w)-f(v))$, by antisymmetry of~$\theta$. Consequently, since~$f \equiv 0$ on~$\partial A$, 
	\[
		\sum_{e \in \E[\overline{A}]} \theta(e) \nabla^{\mathcal{G}} f(e) =  \sum_{v \in A} -2 f(v) \sum_{(w,v) \in \E} \theta(w,v)  
		= \sum_{v \in A} -2 f(v) (\div \theta)(v) \, , 
	\]
	which completes the proof.
\end{proof}

The next lemma states that linear functions lie in the kernel of~$\aDelta$. In particular, like the usual Laplacian, the operator~$\aDelta$ has no drift.
This is not used anywhere in the paper, but is included to give some justification for the choice of conductances in~$\aDelta$.
A similar statement in the special case of sphere packings appears as~\cite[Proposition 6.4]{MR4425348}.
\begin{lemma} \label{lemma:linear-lie-in-kernel}
	For every~$p \in \Rd$ and ~$c \in \R$, denoting~$\ell_p(x) = p \cdot x + c$, we have that $\aDelta \ell_p = 0$. 
\end{lemma}
\begin{proof}
	Let~$v \in \V$ and observe that by definition, 
	\begin{align*}
		\aDelta \ell_p(v) = \sum_{(w,v) \in \E} \frac{\vol_{d-1}(P_w \cap P_v)}{\|w - v\|} (\ell_p(w) - \ell_p(v))
		= 
				\sum_{(w,v) \in \E} \vol_{d-1}(P_w \cap P_v) \left( p \cdot \frac{(w-v)}{\|w - v\|} \right)\, . 
	\end{align*}
	By the (continuum) divergence theorem we have
	\[
	0 = \int_{P_v} \CLap \ell_p  = \int_{\partial P_v} \Cnabla \ell_p \cdot \nu
	= \int_{\partial P_v} p \cdot \nu 
	   \, , 
	\]
	where~$\nu$ denotes the outward pointing normal vector field to~$P_v$. By the definition of an orthogonal tiling, whenever $(w,v) \in \E$ we have $\nu = \frac{w-v}{\|w-v\|}$ on $P_v \cap P_w$. Therefore,
	\[
	\int_{\partial P_v} p \cdot \nu  = 	\sum_{(w,v) \in \E} \vol_{d-1}(P_w \cap P_v) \left( p \cdot \frac{(w-v)}{\|w - v\|} \right)\,.  
	\]
	Combining the previous three displays completes the proof. 
\end{proof}

We next recall the dual variational principle for the discrete Dirichlet problem. A continuum version of this appears in, for example, \cite[Section 8.7, Exercise 13]{evans2022partial}. Also see~\cite[Proposition 4.9]{gurel2020dirichlet}. This is only used in the proof of Proposition~\ref{prop:energy-bound} below.
 
\begin{lemma} \label{lemma:dual-variational-principle}
	Let~$A \subset \V$ be a finite set of vertices, and let~$f: \overline{A} \to \R$ be a function which is
	zero on~$\partial A$. Then, for any vector field~$\theta: \E[\overline{A}] \to \R$ with~$\div \theta = \div (\anabla f)$ on~$A$ we have
	\[
			\sum_{e \in \E[\overline{A}]}  \cond(e)^{-1} |\anabla f(e)|^2 \leq \sum_{e \in \E[\overline{A}]} \cond(e)^{-1}  \theta(e)^2       \, . 
	\]
\end{lemma}
\begin{proof}
	Define~$\theta_* := \anabla f$ and let~$\theta: \E[\overline{A}] \to \R$ be a vector field with~$\div(\theta) = \div(\theta_*)$ on~$A$
	and~$\theta \neq \theta_*$. 
	For~$t \in [0,1]$ define
	\[
	\phi(t) := \sum_{e \in \E[\overline{A}]} \cond(e)^{-1} [ \theta_*(e) + t(\theta(e) - \theta_*(e))]^2 \, ,
	\]
	and observe that it suffices to show that~$\phi$ is strictly increasing on the interval~$[0,1]$. We compute, for~$t \in (0,1)$, 
	\begin{align*}
		\phi'(t) &= 2 \sum_{e \in \E[\overline{A}]} \cond(e)^{-1}  (\theta(e) - \theta_*(e)) [ \theta_*(e) + t(\theta(e) - \theta_*(e))]  \\
		&= 2 \sum_{e \in \E[\overline{A}]}  (\theta(e) - \theta_*(e)) \nabla^{\mathcal{G}} f(e) + 2 t  \sum_{e \in \E[\overline{A}]} \cond(e)^{-1} (\theta(e) - \theta_*(e))^2  \\
		&> 2 \sum_{e \in \E[\overline{A}]}  (\theta(e) - \theta_*(e)) \nabla^{\mathcal{G}} f(e) \qquad \mbox{(since ~$\theta_* \neq \theta$)} \\
		&=  -4\sum_{v \in A}  (\div (\theta - \theta_*))(v) f(v) \qquad \mbox{(Lemma~\ref{lemma:discrete-divergence-theorem})} \\
		&= 0  \, , 
	\end{align*}
	where in the last step we used the assumption~$\div(\theta) = \div(\theta_*)$ on~$A$. 
\end{proof}

\subsection{Geometry}
\begin{figure}
	\includegraphics[width=0.25\textwidth]{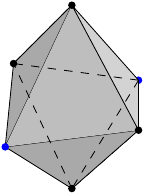}
	\caption{A dual polytope~$\Qe$ in three dimensions; the two vertices~$v$ and~$w$ are in blue and the extremal points of~$H_e$ are in black.}
	\label{fig:threed-qe}
\end{figure}

\begin{figure}
	\includegraphics[width=0.25\textwidth]{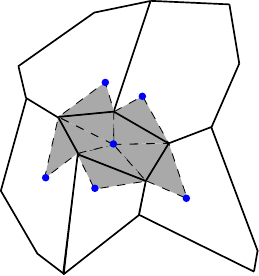}
	\caption{A polytope and its neighbors are outlined in thick black and the vertices are blue dots. For each neighbor, the corresponding~$\Qe$ sets are outlined in dashed lines and filled in dark gray.}
	\label{fig:twod-qe}
\end{figure}

We consider the \emph{dual polytope} associated with the edge~$e = (w,v) \in \E$, 
\begin{equation}
	\label{eq:qeset}
\Qe := \conv( w, P_w \cap P_v) \cup \conv(v, P_w \cap P_v)   \, . 
\end{equation}
See Figures~\ref{fig:threed-qe} and~\ref{fig:twod-qe}.
This set appears in our arguments below and we will need to compute its volume. 
The following elementary identity may be deduced from~\cite[Lemma B.2]{droniou2018gradient} but we include a proof for convenience. 

	\begin{figure}
	\includegraphics[width=0.25\textwidth]{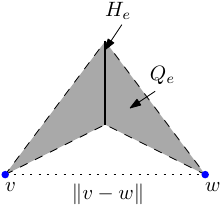}
	\caption{Visual aid to computing the volume of~$\Qe$ in the proof of Lemma~\ref{lemma:geometry-bound}. The points~$v,w$ are in blue, the set~$\Qe$ is in light gray outlined by dashed lines, the facet~$H_e$ is a thick black line, and the line between~$v$ and~$w$ is dotted black.}
	\label{fig:non-gabriel-integrate}
\end{figure}
\begin{lemma} \label{lemma:geometry-bound}
	For each edge~$e = (w,v) \in \E$ we have that 
	\[
	 \|w - v\| \vol_{d-1}(H_e)  = d  \vol_d(\Qe) \, . 
	\]
\end{lemma}
\begin{proof}
	By an orthogonal transformation, we may assume that the facet~$H_e$ lies on the hyperplane~$x_1 = 0$
	and that~$w = (w_1, 0, \ldots, 0)$ and~$v = (-v_1, 0, \ldots, 0)$ for~$v_1, w_1 > 0$. In particular, we have
	\[
	 \conv(w, H_e) = \{ \theta x + (1-\theta) w \quad  \mbox{such that~$x \in H_e$ and~$\theta \in [0,1]$}\} \, . 
	\]
	See Figure~\ref{fig:non-gabriel-integrate}.

	Consequently, the volume may be computed via a change of variables~$\phi: \Rd \to \Rd$ where~$\phi(\theta, x) = (1-\theta) x + \theta w_1$ which has Jacobian determinant~$\theta^{d-1} w_1$ as follows: 
	\[
	\int_{ \conv(w, H_e)} \dvol_{d} = \int_{ \phi([0,1] \times H_e)}  \dvol_{d} =  \int_{0}^{1} \int_{H_e} \theta^{d-1} w_1 \dvol_{d-1} d \theta
	= \frac{1}{d} w_1 \vol_{d-1}(H_e) \, . 
	\]
	Similarly, we have that 
	\[
	\vol_{d}(\conv(v, H_e)) = \frac{1}{d} v_1 \vol_{d-1}(H_e) \, . 
	\]
	The previous two displays yield the claim. 
\end{proof}

We will also need to compare the~$d$-dimensional volumes of sets of polytopes which intersect~$U$.
\begin{lemma} \label{lemma:easy-global-qe-bounds}
	 Recall the definition of $\GCU$ for $U\subset \mathcal{D}$ from~\eqref{eq:graph-closure}. For every bounded set~$U \subset \mathcal{D}$
	\begin{equation}
		\label{eq:easy-bounds-geometric-sets}
	   U \subset \bigcup_{e \in \E[\overline{U}]} \Qe \subset \GCU \, . 
\end{equation}
\end{lemma}

\begin{proof}

The bound~\eqref{eq:easy-bounds-geometric-sets} follows immediately from the fact, by convexity, that dual polytopes cover polytopes,~$P_v \subset \bigcup_{e=(w,v) \in \E} \Qe$, see Figure~\ref{fig:twod-qe}. 
\end{proof}

\subsection{Convergence of the energy} 
Given a subset of vertices~$A$ with~$\overline{A} \subset \V$ and a function~$f: \overline{A} \to \R$, the~\emph{Dirichlet energy}
of~$f$ is defined as 
 \begin{equation} \label{eq:dir.energy}
	\D(f, A) :=  \sum_{e \in \E[\overline{A}]} \cond(e)^{-1} |\anabla f(e)|^2  \, . 
\end{equation} 
When~$U \subset \Rd$, we write
\[
\D(f, U) := \D(f, \V[U]) \,. 
\]
This definition of energy may also be thought of as a discrete Sobolev norm adapted to the geometry of the orthogonal tiling graph, see, \eg,~\cite[Definition 9.3]{eymard2000finite}.

We show the energy of the difference of a continuum harmonic function and its discrete harmonic extension is small. This is well known in the finite volume literature, see for example,~\cite[Theorem 9.3]{eymard2000finite}.
\begin{prop}
	\label{prop:energy-bound}
	Let~$U$ be a bounded domain with~$\overline{U} \subset \mathcal{D}$ and suppose that~$h_C \in C^{2}(\GCU)$ satisfies~$\CLap h_C = 0$ on~$\GCU$ and let~$h_D: \overline{\V[U]} \to \R$ solve 
	\[
	\left\{
	\begin{aligned}
		&  \aDelta h_D = 0 & \mbox{in} & \ \V[U] \,,
		\\
		& h_D = h_C & \mbox{on} & \  \partial \V[U] \, . 
	\end{aligned}
	\right.
	\]
	Then, with
	\[
	\eps := \sup_{v \in \overline{\V[U]}} \diam(P_v)
	\quad \mbox{and} \quad M := \sup_{\GCU} |\Cnabla^2 h_C| 
	\]
	we have the bound
	\begin{equation}
		\label{eq:dirichlet-energy-convergence-bound}
	\D(h_D - h_C, U) \leq 9  M^2 \eps^2 \sum_{e \in \E[\overline{U}]} d \vol_{d}(\Qe)   \, .
	\end{equation}
\end{prop}
\begin{proof}
	Consider the vector field~$\tilde \theta: \E[\overline{U}] \to \R$ defined by 
	\[
	\tilde \theta(w,v) = \int_{P_w \cap P_v} \frac{\Cnabla h_C(y) \cdot (w-v)}{\|w - v\|} \dvol_{d-1}(y)  \qquad \forall (w,v) \in \E[\overline{U}] \, ,
	\]
	and observe that, by the divergence theorem, for every~$v \in \V[U]$ 
	\begin{align*}
	(\div \tilde \theta)(v) &= \sum_{(w,v) \in \E} \tilde \theta(w,v) \\
	&= \sum_{(w,v) \in \E} \int_{P_w \cap P_v} \frac{\Cnabla h_C(y) \cdot (w-v)}{\|w - v\|}  \dvol_{d-1}(y)
	= \int_{P_v}\CLap h_C(x) \dvol_{d}(x)= 0 \, . 
	\end{align*}
	Define the vector field~$\theta: \E[\overline{U}] \to \R$ by
	\[
	\theta(w,v) := \tilde \theta(w,v) - \anabla h_C(w,v)  \quad \forall (w,v) \in \E[\overline{U}] \, . 
	\]
	Let~$f := h_D - h_C$ and observe that by the above two displays and the fact that~$h_D$ is discrete harmonic, we have that on~$\V[U]$, 
	\[
	\div(\theta) = \div(- \anabla h_C) = \div(\anabla h_D - \anabla h_C) =  \div(\anabla f) \, . 
	\]
	Consequently, since~$f$ is zero on~$\partial \V[U]$, we have by Lemma~\ref{lemma:dual-variational-principle} that
	\begin{equation}
		\label{eq:bound-on-d-by-flow}
	\D(f, U) \leq \sum_{e \in \E[\overline{U}]} \cond(e)^{-1} \theta(e)^2 \, , 
\end{equation}
	and so it suffices to bound the sum on the right. By definition of~$\anabla f$ we have 
	\begin{align*}
	\theta(w,v)  
	&= \int_{P_w \cap P_v} \frac{\Cnabla h_C(y) \cdot (w-v) - (h_C(w) - h_C(v))}{\|w - v\|} \dvol_{d-1}(y) \, ,
	\end{align*}
	and so 
	\begin{align}
	&\cond(w,v)^{-1} \theta(w,v)^2  \notag \\
	&= \|w-v\| \vol_{d-1}(P_w \cap P_v) \times  \notag \\
	&\qquad  \left( \frac{1}{\vol_{d-1}(P_w \cap P_v)} \int_{P_w \cap P_v} \frac{\Cnabla h_C(y) \cdot (w-v) - (h_C(w) - h_C(v))}{\|w - v\|} \dvol_{d-1}(y) \right)^2  
	\label{eq:bound-on-conductance}
	\, . 
	\end{align}
	By Taylor's theorem 
	\[
	|h_C(w) - (h_C(v) + \Cnabla h_C(w) \cdot (w - v))| \leq \frac{M}{2} \|w-v\|^2 
	\]	
	and the mean value inequality, 
	\[
	|(\Cnabla h_C(y) - \Cnabla h_C(y')) \cdot (w-v) | \leq M \ \|y-y'\| \  \|w-v\|  \qquad \forall y,y' \in P_w \cup P_v \, . 
	\]
	The previous two displays and the triangle inequality yield 
	\begin{equation}
	\label{eq:bound-on-flux-difference}
	\sup_{y \in P_w \cap P_v} \frac{|\Cnabla h_C(y) \cdot (w-v) - (h_C(w) - h_C(v))|}{\|w - v\|} \leq 3 M \eps \, . 
	\end{equation}	
	Combining~\eqref{eq:bound-on-d-by-flow},~\eqref{eq:bound-on-conductance}, and~\eqref{eq:bound-on-flux-difference} with Lemma~\ref{lemma:geometry-bound} completes the proof.
\end{proof}

\section{Iterating to get a pointwise bound} 
\label{sec:iteration}
In this section we prove Theorem~\ref{theorem:convergence-of-the-dirichlet-problem} under Hypothesis~\ref{item:cont-neighborhood}. We start by converting the energy estimate (Proposition~\ref{prop:energy-bound}) into a pointwise bound on columns (Proposition~\ref{prop:energy-poincare-bound})
using a Poincar\'e-type bound (Lemma~\ref{lemma:ptype-inequality}).

We then observe that this leads, by planarity, to a short proof when~$d=2$ (Theorem~\ref{theorem:2d}).
In higher dimensions, we use Hypotheses~\ref{item:vol} or~\ref{item:diam}
together with an iteration to prove the claims in Theorems~\ref{theorem:generaldim} or~\ref{theorem:generaldim-decay} respectively.

\subsection{Pointwise bound on columns}
\begin{figure}
	\includegraphics[width=0.7\textwidth]{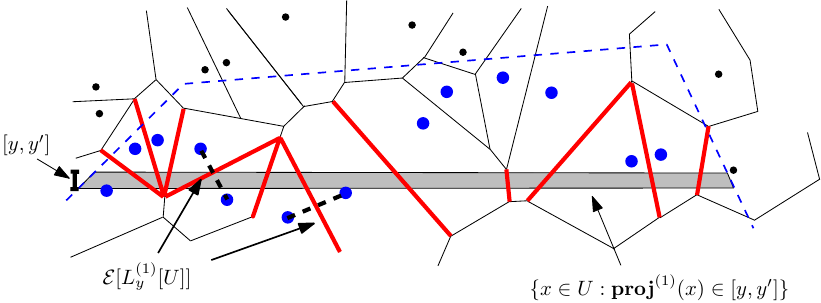}
	 \caption{Part of Figure~\ref{fig:induced-subgraph} is shown along with
		(red) facets which have projections intersecting a small (gray) interval~$[y,y']$. Two edges in~$\E[\Lyone[U]]$ are displayed.} 
			\label{fig:project-lines}
\end{figure}

We record a bound obtained from the proof of the Poincar\'e inequality. The discrete Poincar\'e inequality which this leads
to may be found in~\cite[Lemma 9.1]{eymard2000finite}.

For~$x \in \Rd$ and~$j \in \{1, \ldots, d\}$ we write~$\projj(x)  = (x_1, \ldots, x_{j-1}, x_{j+1}, \ldots, x_d)$ for the orthogonal projection of~$x$ onto~$\R^{d-1}$ obtained by removing the~$j$th coordinate. For a domain~$U \subset \Rd$,~$j \in \{1, \ldots, d\}$, and~$y \in \R^{d-1}$, we write
\begin{equation}
	\begin{aligned}
	\label{eq:line-project}
	\E[\Lyj[U]] &= \{ e \in \E[\overline{U}] : y \in \projj(H_e) \} \\
	\Lyj[U] &= \{v \in \V[U]: v \mbox{ is an endpoint of~$e$ for some~$e \in \E[\Lyj[U]]$} \} \, ,
	\end{aligned} 
\end{equation}
for the set of edges (resp. endpoints of edges) adjacent to~$U$ with projected facets which intersect~$y$ --- see Figure~\ref{fig:project-lines}. 
Note that there can be edges in~$\E[\Lyj[U]]$ whose projections do not contain $y$.

\begin{lemma}[Poincar\'e-type inequality]
	\label{lemma:ptype-inequality}
	For every bounded domain~$U$ with~$\overline{U} \subset \mathcal{D}$, coordinate~$j \in \{1, \ldots, d\}$, and function~$f: \V[U] \to \R$ which is zero on~$\partial \V[U]$ we have for every~$k > 0$
	\[
	\vol_{d-1}\left( y \in \projj(U): \max_{v \in \Lyj[U]} |f(v)| > k    \left( \D(f,U) \sum_{e \in \E[\overline{U}]} d   \vol_{d}(\Qe) \right)^{\! \!  \! \!\nicefrac12} \right) \leq k^{-1}  \, . 
	\]
\end{lemma}
\begin{proof}

	By interchanging summation and integration, 
	\[
		\int_{\projj(U)} \sum_{e \in \E[\Lyj[U]]} |\nabla^{\mathcal{G}} f(e)| \dvol_{d-1}(y)
		 \leq \sum_{e \in \E[\overline{U}]} |\nabla^{\mathcal{G}} f(e)| \vol_{d-1}(H_e)  \, . 
	\]
	Consequently, by the Cauchy–Schwarz inequality, 
	\begin{align*}
		&\int_{\projj(U)} \sum_{e \in \E[\Lyj[U]]} |\nabla^{\mathcal{G}} f(e)| \dvol_{d-1}(y) \\
		&\leq  \sqrt{\sum_{e \in \E[\overline{U}]} |\nabla^{\mathcal{G}} f(e)|^2 \cond(e)} \sqrt{ \sum_{e \in \E[\overline{U}]}    \cond^{-1}(e) \vol_{d-1}(H_e)^2	} \\
		&=  \sqrt{\sum_{e \in \E[\overline{U}]}  |\nabla^{\mathcal{G}} f(e)|^2 \cond(e)} \sqrt{ \sum_{e \in \E[\overline{U}]}   d \vol_{d}(\Qe)	} \, , 
	\end{align*}
	where in the last equality we used Lemma~\ref{lemma:geometry-bound}. This implies
	\begin{equation}
		\label{eq:bound-on-sums-of-gradients}
		\int_{\projj(U)} \sum_{e \in \E[\Lyj[U]]} |\nabla^{\mathcal{G}} f(e)| \dvol_{d-1}(y) \leq   \D(f,U)^{\nicefrac12}   \left( \sum_{e \in \E[\overline{U}]} d   \vol_{d}(\Qe)  \right)^{\! \!  \! \!\nicefrac12} \, . 
	\end{equation}
	Observe that for~$\vol_{d-1}$-almost every~$y \in \projj(U)$, since~$f$ is zero on the boundary, by summing the gradient,
	\[
	\sum_{e \in \E[\Lyj[U]]} |\nabla^{\mathcal{G}} f(e)| < \delta \implies \max_{v \in \Lyj[U]} |f(v)| < \delta  \qquad \forall \delta > 0 \, . 
	\]
	Here we are using the fact that for~$\vol_{d-1}$-almost every~$y \in \projj(U)$ the vertices in~$\Lyj[U]$ are connected via the edges in $\E[\Lyj[U]]$.
	
	By the above two displays and Chebyshev's inequality, we have that, for all~$\delta> 0$
	\[
	\vol_{d-1} \left( y \in \projj(U): \max_{v \in \Lyj[U]} |f(v)| > \delta \right) \leq \delta^{-1}  \D(f,U)^{\nicefrac12}   \left( \sum_{e \in \E[\overline{U}]} d   \vol_{d}(\Qe)  \right)^{\! \!  \! \!\nicefrac12}  \, . 
	\]
	This completes the proof. 
\end{proof}

By combining the above inequality together with Proposition~\ref{prop:energy-bound} we obtain a pointwise bound on the 
difference between the discrete and continuous harmonic functions on an arbitrarily dense set of columns. 
\begin{prop}[Closeness on columns]
	\label{prop:energy-poincare-bound}
	Let~$U$ be a bounded domain with~$\overline{U} \subset \mathcal{D}$ and suppose that~$h_C \in C^{2}(\GCU)$ satisfies~$\CLap h_C = 0$ on~$\GCU$ and let~$h_D: \overline{\V[U]} \to \R$ satisfy 
	\[
	\left\{
	\begin{aligned}
		&  \aDelta h_D = 0 & \mbox{in} & \ \V[U] \,,
		\\
		& h_D = h_C & \mbox{on} & \  \partial \V[U] \, . 
	\end{aligned}
	\right.
	\]
	Then, with
	\[
	\eps := \sup_{v \in \overline{\V[U]}} \diam(P_v)
	\quad \mbox{and} \quad M := \sup_{\GCU} |\Cnabla^2 h_C|
	\]
	we have for every~$k > 0$ and~$j \in \{1, \ldots, d\}$,
	\[
	\vol_{d-1} \left( y \in \projj(U): \max_{v \in \Lyj[U]} |h_D(v) - h_C(v)| > 3 k M \eps \sum_{e \in \E[\overline{U}]} d   \vol_{d}(\Qe)  \right) \leq k^{-1}  \, . 
	\]
\end{prop}
\begin{proof}
	This is immediate from Lemma~\ref{lemma:ptype-inequality} and Proposition~\ref{prop:energy-bound}. 
\end{proof}

%

\subsection{Two dimensions}
We give a short proof of a strengthened version of the main result in~\cite{gurel2020dirichlet}.
Roughly, when~$d=2$, by taking~$k$ large enough, Proposition~\ref{prop:energy-poincare-bound} guarantees that there is a fine grid of ``good lines'' upon which the discrete and continuous harmonic functions~$h_D$ and~$h_C$ are close. On every rectangle formed by the grid, by smoothness,~$h_C$ is essentially constant. Thus, by the maximum principle,~$h_D$ is close to~$h_C$ everywhere. 
\begin{figure}
	\includegraphics[width=0.5\textwidth]{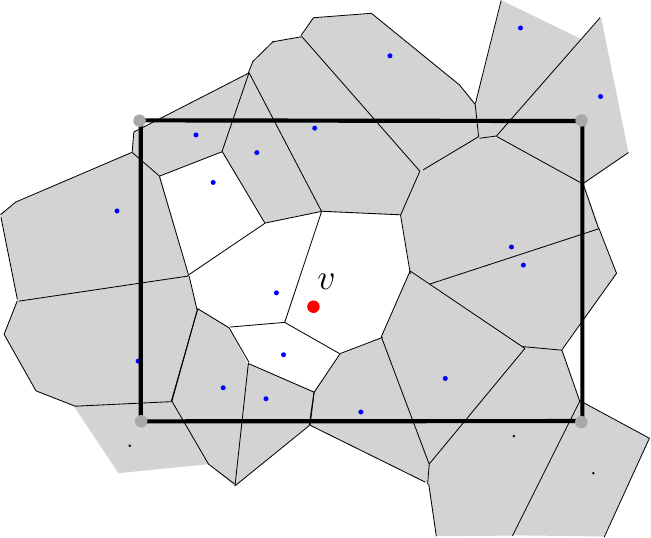}
	\caption{Illustration of the proof of Theorem~\ref{theorem:2d}. The square~$S$ is outlined in black, the cells for which the discrete harmonic function is close to the continuum one are in light gray. }	\label{fig:small-good-square}
\end{figure}

\begin{theorem}
	\label{theorem:2d}
	Let~$U$ be a bounded domain with~$\overline{U} \subset \mathcal{D} \subset \R^2$ and suppose~$h_C \in C^{2}(\GCU)$ satisfies~$\CLap h_C = 0$ on~$\GCU$ and let~$h_D: \overline{\V[U]} \to \R$ satisfy 
	\[
	\left\{
	\begin{aligned}
		&  \aDelta h_D = 0 & \mbox{in} & \ \V[U] \,,
		\\
		& h_D = h_C & \mbox{on} & \  \partial \V[U] \, . 
	\end{aligned}
	\right.
	\]
	Then, with
	\[
	\eps := \sup_{v \in \overline{\V[U]}} \diam(P_v) \quad L := \sup_{\GCU} |\Cnabla h_C| \quad \mbox{and} \quad M := \sup_{\GCU} |\Cnabla^2 h_C| \, , 
	\]
	we have the pointwise bound
	\[
	\sup_{v \in \V[U]} | h_D(v) - h_C(v) | \leq \inf_{k >0} \left( 6 k M \eps \vol_{d}(\GCU) + 2 (\sqrt{2} k^{-1} + \eps) L \right) \, .  
	\]
In particular, there exists~$C= C(U, M, L) <\infty$ so that
	\[
	\sup_{v \in \V[U]} | h_D(v) - h_C(v) |\leq  C \eps^{\nicefrac12}  \, . 
	\]
\end{theorem}
\begin{proof}
	 By Proposition~\ref{prop:energy-poincare-bound} we have, for each~$j \in \{1,2\}$ and every~$k > 0$, 
	\begin{equation}
		\label{eq:bound-on-measure-of-lines}
	\left| \underbrace{ \left\{y \in \projj(U): \max_{v \in \Lyj[U]} (h_D(v) - h_C(v)) > 6 k   M \eps \vol_{d}(\GCU) \right\} }_{=:Y^{(j)}} \right| \leq k^{-1}  \, , 
	\end{equation}
	where we also used~\eqref{eq:easy-bounds-geometric-sets}.
	Let~$v \in \V(U)$ be given and suppose~$h_D(v) - h_C(v) > 6 k M \eps \vol_{d}(\GCU)$. 
	
	We claim that the above display implies that there exists a rectangle~$S$, with~$v \in S$, such that
	\begin{equation}
		\label{eq:good-square-exists}
	\diam(S) \leq 2 \sqrt{2} k^{-1}  \quad \mbox{and} \max_{w \in \partial \V[S \cap U]  } (h_D(w) - h_C(w)) < 6 k  M \eps \vol_{d}(\GCU) \quad \forall k > 0 \, ,  
	\end{equation}
	see Figure~\ref{fig:small-good-square}. Indeed, let~$S = S^{(2)} \times S^{(1)}$ where $S^{(j)}$ is the largest interval in~${Y}^{(j)}$ containing~$v_i$. By~\eqref{eq:bound-on-measure-of-lines},~$\diam(S) \leq 2 \sqrt{2} k^{-1}$.  Also, by maximality of each interval, if~$w \in \partial \V[S \cap U] \setminus \partial \V[U]$, then~$(h_D(w) - h_C(w)) < 6 k  M \eps \vol_{d}(\GCU)$. If~$w \in \partial \V[U]$, then~$h_D(w) - h_C(w) = 0$. This establishes~\eqref{eq:good-square-exists}. 
	 
	 We also have by the mean value theorem that 
	\[
	\sup_{x, y \in \overline{\V[S \cap U]}} |h_C(x) - h_C(y) | \leq  (\diam(S) + 2 \eps) L \, .  
	\]
	Since~$v \in S$, by the maximum principle~$h_D(v) \leq h_D(w)$ for some vertex~$w \in \partial \V[S \cap U]$. Consequently, by the above display,~\eqref{eq:good-square-exists} and the triangle inequality we have, for all~$k > 0$, 
	\begin{align*}
	h_D(v) - h_C(v) &\leq h_D(w) - h_C(v)   \\
	&\leq h_D(w) - h_C(w) + 2 (\sqrt{2} k^{-1} + \eps) L \\
	&\leq 6 k M \eps \vol_{d}(\GCU) + 2 (\sqrt{2} k^{-1} + \eps) L \, . 
	\end{align*}
	A symmetric argument establishes the bound in the other direction. The ``in particular'' clause follows by choosing~$k = \eps^{-\nicefrac12}$. 
\end{proof}

\begin{figure}[!h]
\fbox{\includegraphics[width=0.25\textwidth]{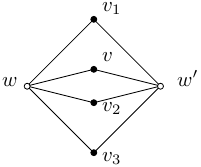}} \qquad
\fbox{\includegraphics[width=0.25\textwidth]{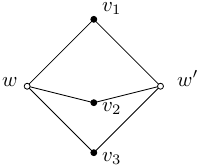}} \qquad
\fbox{\includegraphics[width=0.25\textwidth]{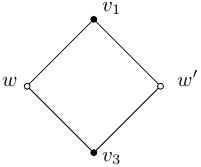}}
\caption{Series reduction of an orthodiagonal map; from left to right the vertices~$v$ and~$v_2$ in~$\V^{\bullet}$ of degree two are removed. The same labeling scheme as Figure~\ref{fig:orthogonal-diagonal} is used.}
\label{fig:series-reduction}
\end{figure}

\begin{figure}
\includegraphics[width=0.35\textwidth]{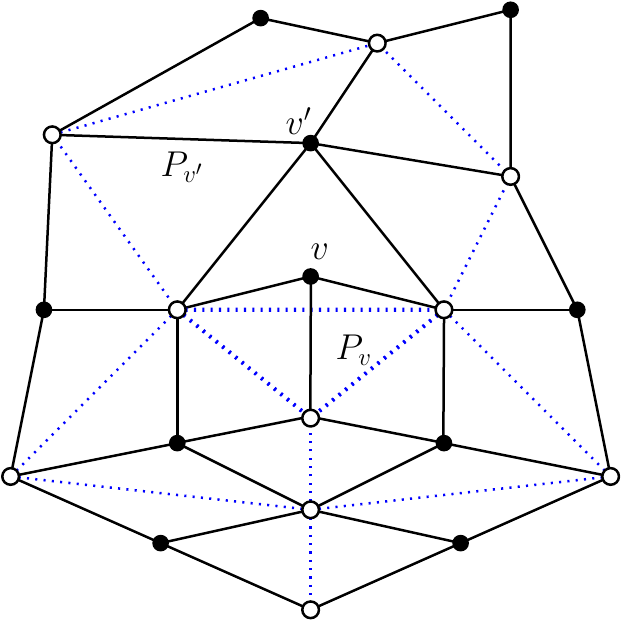}
\caption{A vertex~$v$ in an orthodiagonal map which is not contained in its corresponding polytope~$P_v$
but contained in a polytope of a neighbor~$P_{v'}$. The same labeling scheme as Figure~\ref{fig:orthogonal-diagonal} is used.}
\label{fig:concave-complication}
\end{figure}

\begin{remark}
\label{remark:orthodiagonalmaps}
Recall the definition of orthodiagonal map from Section~\ref{subsection:orthodiagonal}. In the case some quadrilaterals in~$\G$ are nonconvex, the construction there may fail 
to produce an orthogonal tiling. In this remark, we explain how our proof still applies to this case.

We first observe that we can use electric network laws to reduce to the case where all interior vertices in~$\G^{\bullet}$ and~$\G^{\circ}$ have degree at least three. An interior vertex in~$\G$ must be incident to at least two inner faces, so every interior vertex in~$\G^{\bullet}$ has degree at least two. An interior vertex~$v$ in~$\G^{\bullet}$ with exactly two neighbors~$v_1, v_2 \in \G^{\bullet}$ must share an edge in~$\E$ with exactly two vertices~$w, w' \in \V^{\circ}$. Thus, by the series law, any function~$f: \V^{\bullet} \to \R$ which is~$\mathbf{a}$-discrete harmonic must also be~$\mathbf{a}$-discrete harmonic on the graph where the vertex~$v$ and its two incident edges in~$\G$ is removed --- see Figure~\ref{fig:series-reduction}. Indeed, since~$v_1$ and~$v_2$ are neighbors of~$v$ and~$w-w'$ is orthogonal to~$v-v_1$ and~$v - v_2$, the vertex~$v$ lies on the line segment between~$v_1$ and~$v_2$; therefore
\[
\frac{\|v_1 - v_2\|}{\|w - w'\|} = \frac{\|v_1 - v\| + \|v - v_2\|}{\|w - w'\|}  = \frac{1}{\mathbf{a}(v_1, v)} + \frac{1}{\mathbf{a}(v_2, v)} \, . 
\]
Similarly, by the parallel law, we may reduce to the case there are no internal vertices in~$V^{\circ}$ of degree less than three. 
The result of Theorem~\ref{theorem:2d} on the reduced graph implies the result on the non-reduced graph. Indeed, arguing with the maximum principle for discrete harmonic functions and the mean-value property as in the proof of Theorem~\ref{theorem:2d}, if a discrete harmonic function is close to a continuum harmonic function on the reduced graph, then it must be close to the same continuum harmonic function on the deleted vertices.

Every vertex~$v \in \V^{\bullet}$ is now associated to a well-defined strictly convex polygon~$P_v$ with edges given by the opposite diagonals of the quadrilaterals
containing~$v$. The polygons~$\{P_v\}_{v \in \V^{\bullet}}$ have pairwise disjoint interiors. Every vertex~$v \in \V^{\bullet}$ 
is either contained in the interior of~$P_v$ or in the interior of~$P_{v'}$ where~$v' \in \V^{\bullet}$ is a neighbor of~$v$ --- see Figure~\ref{fig:concave-complication}.

We also define, for each edge~$e = (v,v') \in \mathcal{E}^{\bullet}$, the set~$Q_e$ to be the quadrilateral in~$\G$ with extremal vertices~$v$ and~$v'$. Note for convex quadrilaterals, this coincides with~$\Qe$ defined in~\eqref{eq:qeset}. 

With these two modified definitions, the proof of Theorem~\ref{theorem:2d} proceeds in the exact same way. 
\end{remark}

\subsection{All dimensions with smallness hypotheses}
We now iterate Proposition~\ref{prop:energy-poincare-bound} 
together with an assumption on tiny polytopes to go from smallness on columns to a pointwise bound.  The following bound on the measure of dual polytopes which correspond to an edge with large gradient will be used.

\begin{lemma}[Weighted gradient-energy estimate]
	\label{lemma:gradient-energy-estimate}
	For every bounded set of vertices~$A$ with~$\overline A \subset \V$ and function~$f: \overline{A} \to \R$  we have
	the bound
	\[
	 \sum_{e \in \E[\overline{A}]} d \vol_{d}(\Qe)  1\{ |\nabla^{\mathcal{G}} f(e)| > \delta\}  \leq \D(f, A) 	 \qquad \forall \delta \geq \sup_{e \in \E[\overline{A}]} \|e\| \, . 
	\]
\end{lemma}
\begin{proof}
	We compute, using the assumption~$\delta \geq \sup_{e \in \E[\overline{A}]} \|e\|$, 
	\begin{align*}
		\D(f, A)  =  \sum_{e \in \E[\overline{A}]} \frac{\vol_{d-1}(H_e)}{\|e\|} |\nabla^{\mathcal{G}} f(e)|^2  
		&\geq \sum_{e \in \E[\overline{A}]} \frac{\vol_{d-1}(H_e)|}{\|e\|} \delta^2 1\{ |\nabla^{\mathcal{G}} f(e)| > \delta\} \\
		&\geq \sum_{e \in \E[\overline{A}]} \vol_{d-1}(H_e) \|e\|   1\{ |\nabla^{\mathcal{G}} f(e)| > \delta\}  \\
		&= \sum_{e \in \E[\overline{A}]} d \vol_{d}(\Qe)  1\{ |\nabla^{\mathcal{G}} f(e)| > \delta\}  \, , 
	\end{align*}
	where in the last step we used Lemma~\ref{lemma:geometry-bound}. 
\end{proof}

\begin{theorem}
		\label{theorem:generaldim}
	Let~$U$ be a bounded domain with~$\overline{U} \subset \mathcal{D} \subset \R^d$ and suppose~$h_C \in C^{2}(\GCU)$ satisfies~$\CLap h_C = 0$ on~$\GCU$ and let~$h_D: \overline{\V[U]} \to \R$ satisfy 
	\[
	\left\{
	\begin{aligned}
		&  \aDelta h_D = 0 & \mbox{in} & \ \V[U] \,,
		\\
		& h_D = h_C & \mbox{on} & \  \partial \V[U] \, . 
	\end{aligned}
	\right.
	\]
	Then, there exists~$C= C(U) < \infty$ such that, with
	\[
	\eps := \sup_{v \in \overline{\V[U]}} \diam(P_v) \quad L := \sup_{\GCU} |\Cnabla h_C| \quad \mbox{and} \quad M := \sup_{\GCU} |\Cnabla^2 h_C| \vee 1 \, , 
	\]
	and for any fixed~$K \geq 1$, which can depend on~$\eps$,
	\begin{equation}
		\begin{aligned}
		\label{eq:small-volume-assumption}
		\mathcal{A} := \{ v \in \overline{\V[U]} : \vol_{d}(P_v) < \eps^{K}  \}  \quad \mbox{and} \quad
		D := \sup_{\mbox{components $Q \subset \mathcal{A}$}} {\sum_{v \in Q} \diam(P_v)}  \, , 
		\end{aligned}
	\end{equation}
	where the supremum is over the connected components of the subgraph induced by~$\mathcal{A}$, we have, assuming~{$\eps  < (C M)^{-1}$},
	\begin{equation} 
		\label{eq:full-bound}
{	 \sup_{v \in \V[U]} | h_D(v) - h_C(v) | \leq C M K \eps(\log \eps^{-1}) + D L  \, .}
	\end{equation}
\end{theorem}

	Note that if for some large~$K$ the set~$\mathcal{A}$ in~\eqref{eq:small-volume-assumption} is empty, then the second error term 
	in~\eqref{eq:full-bound} is zero. 
\begin{proof}[Proof of Theorem~\ref{theorem:generaldim}]
	We consider a sequence of functions~$\{h_i\}_{i \geq 0}$ on~$\V[U]$ which converges to~$h_C$ as~$i \to \infty$. 
	Each such function will be the ``discrete harmonic extension'' of~$h_C$ to a subset~$S_i \subset \V[U]$.
	In the first step we define the sequence. In the second step we bound~$|h_D - h_C|$
	in terms of~$|h_i - h_D|$ and the diameter of the components of~$S_i$ using the maximum principle and the continuity of $h_C$.  In the third step, we estimate this diameter using~\eqref{eq:small-volume-assumption} and Proposition~\ref{prop:energy-poincare-bound}. 
	\smallskip

	\emph{Step 1.} We define the iteration. For a set of vertices~$V$ with~$V \subset \V[U]$ we denote the harmonic extension of 
	~$h_C$ to~$V$ by~$\mathcal{H}(V)$, \ie, $\mathcal{H}(V) = g$ where~$g$ is the unique solution~$g: \overline{V} \to \R$ of 
	\[
	\left\{
	\begin{aligned}
		&  \aDelta g = 0 & \mbox{in} & \ V \,,
		\\
		& g = h_C & \mbox{on} &  \ \overline{\V[U]} \setminus V  \, . 
	\end{aligned}
	\right.
	\]
	{Fix a large constant~$A > 1$, depending on~$U$, to be determined below. }
	Start with~$h_0 := h_D$ and recursively set, for~$i \geq 0$, 
	\[
	h_{i+1} = \mathcal{H}(S_i)
	\]
	where~$S_0 := \V[U]$ and
	\begin{equation}
		\label{eq:def-of-bad-set}
		\begin{aligned}
				N_{i+1} &:= \{ v \in \V[U] : |(h_{i} - h_C)(v)| > {A M \eps}    \}  \, ,  \\
			S_{i+1} &:=  N_{i+1}  \setminus \{ v \in N_{i+1} : \exists w \in \partial N_{i+1} \, \, \mbox{with} \, \, |\nabla^{\mathcal{G}}(h_i - h_C)(v,w)| \leq 2 \eps \}    \, .
		\end{aligned}
	\end{equation}

	 That is,~$S_{i+1}$ is defined to be a slightly trimmed subset of~$N_{i+1}$, which itself is a subset of~$S_i$ where~$|h_{i}-h_C|$ is large.
	 The trimming is used in~\eqref{eq:bounding-polytope-volume-iteration} below so that Lemma~\ref{lemma:gradient-energy-estimate} can control the volume of polytopes on the boundary of~$N_{i+1}$ which remain. Note that~$h_0 = h_1 = h_D$. 
	\smallskip

	\emph{Step 2.} We bound the difference between~$h_C$ and~$h_D$ using the sequence defined in Step 1.
	By definition of~$S_{i}$, we have that 
	\[
	|h_i - h_C| \leq (A M+2) \eps \quad \mbox{on} \ \V[U] \setminus S_{i+1}  \qquad \forall i \geq 0 \, . 
	\]
	Also, since the sets are necessarily nested,~$S_{i+1} \subset S_i$, the maximum principle yields that for every~$i \geq 0$,
\[
	\max_{\V[U]} |h_{i+1} - h_{i} | = \max_{S_{i}} |h_{i+1} - h_{i} | = 
	 \max_{\partial S_{i}} |h_{i+1} - h_{i} | = \max_{\partial S_{i} } |h_C - h_i| \leq 
	 (A M+2) \eps \, . 
\]
By the previous two displays and telescoping, we have 
	\begin{equation}
		\label{eq:telescoping-bound}
		|h_D - h_C|  \leq  {i (A M+2) \eps} 
		 \quad  \mbox{on}   \ \overline{\V[U]} \setminus S_i  \quad \forall i \geq 0 \, . 
	\end{equation}
	 By the mean value theorem, 
	 for each component~$Q$ of~$S_i$, 
	\[
	\sup_{x, y \in \overline{Q}} |h_C(x) - h_C(y)| \leq {\sum_{v \in Q} \diam(P_v)  L}
	\]
	and, for each vertex~$v \in Q$, by the maximum principle~$h_D(v) \leq h_D(w)$ for some vertex~$w \in \partial Q$. Consequently, by the above two displays and the triangle inequality, for each such component~$Q$ and vertex~$v \in Q$, there is a vertex~$w \in \partial Q$ such that
	\begin{align*}
		h_D(v) - h_C(v) &\leq h_D(w) - h_C(v)   \\
		&\leq h_D(w) - h_C(w) +  {\sum_{v \in Q} \diam(P_v)  L} \\
		&\leq i (A M+2) \eps  +  {\sum_{v \in Q} \diam(P_v)  L} \, . 
	\end{align*}
	Repeating this for the other direction shows that
	\begin{equation}
		\label{eq:bound-which-includes-diameter}
		\sup_{v \in \V[U]}|h_D(v) - h_C(v)|  \leq  {i (A M+2) \eps}  + \left( \sup_{\mbox{components $Q \subset S_i$}} {\sum_{v \in Q} \diam(P_v)}  \right) L  \quad \forall i \geq 0 \, . 
	\end{equation}

	\smallskip

	\emph{Step 3.} 	We show  that the term in the parentheses on the right of~\eqref{eq:bound-which-includes-diameter} is bounded above by~$D$ for a sufficiently large choice of~$i$.  For each~$i \geq 0$, let  
	\[
	U_i := \bigcup_{v \in S_i} P_v \quad \mbox{and} \quad Y_i := \{y \in \projone(U_i): \max_{v \in \Lyone[U_i]} |(h_i - h_C)(v)| >  	{A M \eps} \} \, , 
	\]   	
	where~$\projone$ and~$\Lyone$ are as in~\eqref{eq:line-project}.
	By Proposition~\ref{prop:energy-poincare-bound} with parameter choice of~$k = (3A^{-1} \sum_{e \in \E[\overline{U_i}]} d   \vol_{d}(\Qe))^{-1}$ and~$U$ replaced by~$U_i$ we have, for every~$i \geq 0$,
	\begin{equation*}
	 	\vol_{d-1} \left( Y_i  \right) \leq  {3 A^{-1} \sum_{e \in {\E[\overline{U_i}]}} d   \vol_{d}(\Qe)} \, .
	\end{equation*}
	Using the definitions of $U_i$ and $S_i$, we may bound the sum on the right as  
	\begin{equation}
	\label{eq:bounding-polytope-volume-iteration}
	\sum_{e \in {\E[\overline{U}_i]}} d   \vol_{d}(\Qe) \leq \underbrace{\sum_{e \in \E[N_i]} d \vol_{d}(\Qe)}_{\leq d \vol_{d-1}(Y_{i-1}) \diam(U)} + 
	\underbrace{\sum_{e \in \partial \E[N_i]} d \vol_{d}(\Qe) 1\{ |\nabla^{\mathcal{G}} (h_{i-1} - h_C)(e)| > 2 \eps \}}_{\leq \D(h_{i-1} - h_C, U_{i-1})} \, . 
	\end{equation}
	The bound for the first term follows since, for each edge~$(w,v) = e \in \E[N_i]$,  as~$v,w \in N_i$, the projection of~$P_w \cup P_v$ is contained in~$Y_{i-1}$ and hence the dual polytope~$\Qe$ is completely contained in a cylinder 
	with base~$Y_{i-1}$ and height~$\diam(U)$. The second term is bounded by  Lemma~\ref{lemma:gradient-energy-estimate} applied with~$f := h_{i-1} - h_C$
	and~$A = U_{i-1}$. 
	
	Combining the above two displays with Proposition~\ref{prop:energy-bound} to bound~$\D(h_{i} - h_C, U_{i})$ yields, for every~$i \geq 0$

	\begin{equation*}
	\label{eq:vol-bound-by-qe}
	\begin{aligned}
	\vol_{d-1} \left( Y_{i+1}  \right)  &\leq {3 A^{-1}} \sum_{e \in {\E[\overline{U_{i+1}}]}} d   \vol_{d}(\Qe) \\
	&\leq {3 A^{-1}} d \diam(U) \vol_{d-1} \left( Y_{i}  \right)
	+ {27 A^{-1} M^2 \eps^{2}} \sum_{e \in \E[\overline{U_{i}}]} d   \vol_{d}(\Qe)  \, , 
	\end{aligned}
\end{equation*}
	and so, if we define 
	\[
	f(i) :=  \vol_{d-1}(Y_i) \vee  {3 A^{-1}} \sum_{e \in \E[\overline{U_i}]} d   \vol_{d}(\Qe)
	\]
	we have 
	\[
	f(i+1) \leq {3 A^{-1}} d \diam(U) f(i) + 9 M^2 \eps^2 f(i) \, . 
	\]
	Assuming~$\eps < 2^{-10} A^{-1} M^{-1}$, the above display implies that there is a constant~$C=C(U) < \infty$ such that
	\[
	f(i+1) \leq C {A^{-1}} f(i) \qquad \forall i \geq 1 \, 
	\]
	and so, using~$f(0) \leq C$, and picking~$A$ sufficiently large, depending on~$U$, so that~$C A^{-1} < \frac14$, we have
	\begin{equation}
		\label{eq:geometric-bound-during-iteration}
	{f(i+1) \leq 2^{-i} \qquad \forall i \geq 1} \, . 
	\end{equation}
	Since, for every~$v \in S_i$, the polytope~$P_v \subset \bigcup_{e \in \E[U_i]} \Qe$, this shows that 
	\[
	\vol_{d}(P_v) \leq  2^{-i} \qquad \forall v \in S_{i+1} \, ,  \quad \forall i \geq 1 \, . 
	\]
	Consequently, by~\eqref{eq:small-volume-assumption} each component of~$S_{C K \log \eps^{-1}}$ must have the sum of the diameters of its corresponding polytopes bounded by~$D$. This together with~\eqref{eq:bound-which-includes-diameter} completes the proof. 
\end{proof}

We modify the previous proof to replace the assumption~\eqref{eq:small-volume-assumption} by one
where the diameter of every polytope is bounded by some (possibly tiny) power of the volume of its dual polytopes. 
\begin{theorem}
	\label{theorem:generaldim-decay}
	Let~$U$ be a bounded domain with~$\overline{U} \subset \mathcal{D} \subset \R^d$ and suppose~$h_C \in C^{2}(\GCU)$ satisfies~$\CLap h_C = 0$ on~$\GCU$ and let~$h_D: \overline{\V[U]} \to \R$ satisfy 
	\[
	\left\{
	\begin{aligned}
		&  \aDelta h_D = 0 & \mbox{in} & \ \V[U] \,,
		\\
		& h_D = h_C & \mbox{on} & \  \partial \V[U] \, . 
	\end{aligned}
	\right.
	\]
	Then with
	\[
	\eps := \sup_{v \in \overline{\V[U]}} \diam(P_v) \quad \quad M := \sup_{\GCU} |\nabla^2 h_C| \vee 1 \, , 
	\]
	and under the assumption that there exists some~$F, \alpha > 0$ such that
	\begin{equation}
		\label{eq:decay-of-stuff}
		\diam(P_v) \leq  F \sup_{e = (w,v) \in \E} \vol_{d}(\Qe)^{\alpha} \qquad \forall v \in \V[U] \, ,  
	\end{equation}
	we have that there exists~$C=C(U, F, \alpha) < \infty$ such that, assuming~{$\eps  < (C M)^{-1}$},
	\begin{equation} 
		\label{eq:full-bound-decay-of-stuff}
		\sup_{v \in \V[U]} | h_D(v) - h_C(v) | \leq {C M \eps (\log \eps^{-1})}  \, . 
	\end{equation}
\end{theorem}
\begin{proof}
	The proof follows the same outline as the previous one but utilizes the observation that the diameter of each polytope  
	in the ``bad set'' should decrease in every step. To avoid repetition, we use the same notation as the previous proof. 
	Objects which are nearly identical to those in the previous proof 
	are denoted with a prime, those which are new have different letters; a central role will be played by the decreasing
	sequence~$\{\delta_i\}_{i \geq 0}$ which will describe the longest edge in the bad set. 
	The constant~$C$ below will be allowed to depend on~$U,F,\alpha$ and may change from line to line. 
	{Fix~$A \geq 1$, which depends only on~$U$, as in~\eqref{eq:def-of-bad-set}. }
	\smallskip

	Start with~$h_0' := h_D$,~$\delta_0 := \eps$,~$S'_0 := \V[U]$ and recursively set, for~$i \geq 0$, 
	\begin{equation}
		\label{eq:def-of-other-bad-set}
		\begin{aligned}
			h'_{i+1} &:=  \mathcal{H}(S'_{i}) \, ,   \\
			N'_{i+1} &:= \{ v \in \V[U] : |(h'_{i} - h_C)(v)| > {A M \delta_i}    \}  \, ,  \\
			S'_{i+1} &:=  N'_{i+1}  \setminus \{ v \in N'_{i+1} : \exists w \in \partial N'_{i+1} \, \, \mbox{with} \, \, |\nabla^{\mathcal{G}}(h'_i - h_C)(v,w)| \leq 2 \delta_i \}    \, , \\
			\delta_{i+1} &:= \sup_{v \in S'_{i+1}} \diam(P_v) \, . 
		\end{aligned}
	\end{equation}
	Following the argument leading to~\eqref{eq:telescoping-bound},  we have
	\begin{equation}
		\label{eq:sum-of-deltaj-bound}
		|h_D - h_C| \leq C {A} M \sum_{j=1}^{i} {\delta_j}  \quad  \mbox{on}   \ \overline{\V[U]} \setminus S'_{i}  \quad \forall i \geq 0 \, . 
	\end{equation}
	Now, instead of using~\eqref{eq:bound-which-includes-diameter} and arguing that the components of~$S'_{i}$ have a small diameter for~$i$ large, we will allow~$i$ to go to infinity in the above display
	and control how quickly~$\delta_j$ goes to zero. 
	
	Since~$\delta_j$ is defined as the smallest diameter of a polytope in~$S'_{j}$, we may replace every instance of~$\eps$ in Step 3 of the previous proof by~$\delta_j$.  Denote by~$U'_j := \bigcup_{v \in S'_{j}} P_v$ so that if we define, in place of~$f(i)$, 
	\[
	f'(i) :=  {C A^{-1}} \sum_{e \in \E[\overline{U'_i}]} d   \vol_{d}(\Qe) \, , 
	\]
	we have,  as in~\eqref{eq:geometric-bound-during-iteration}, 
	\[
	f'(i+1) \leq  {2^{-i}}  \, . 
	\]
	We rewrite the bound on the sum as a bound on the maximum,
	\[
	\left({C A} f'(i) \right)^{\alpha} \geq \left( \sum_{e \in \E[\overline{U'_i}]}  \vol_{d}(\Qe) \right)^{\alpha} \geq \sup_{e \in \E[\overline{U'_i}]} \vol_{d}(\Qe)^{\alpha} \, . 
	\]
	The above two displays together with~\eqref{eq:decay-of-stuff} then imply a bound on~$\delta_i$, 
		\[
		\sup_{v \in S'_{i}} \diam(P_v) \leq C A^{\alpha}  f'(i)^{\alpha}\leq  C 2^{-\alpha i} \, , 
		\]
		where in the last step we absorbed~$A$ into~$C$.

	Since~$\delta_i \leq \eps$ for all~$i$, this shows that 
	\[
	{\delta_i \leq \eps \wedge (C 2^{- \alpha  i}) \quad \forall i \geq 0 \, . }
	\]
	Since the right side goes to zero as~$i \to \infty$, and there are only finitely many polytopes in~$U$, the set~$S'_{i}$ must empty
	for large enough~$i$. Thus, taking $K \in \mathbb N$ and plugging the above display into~\eqref{eq:sum-of-deltaj-bound} and sending~$i \to \infty$ {yields 
\begin{align*}
	|h_D - h_C|  &\leq \inf_{K >0}   \left( C M K \eps + C M \sum_{j=K}^{\infty}    2^{-\alpha j} \right) \\
	&\leq \inf_{K >0}   \left( C M K \eps + C M 2^{-\alpha K} \right) \\
	&\leq C M \eps (\log \eps^{-1})   \quad \mbox{(after choosing~$K = C (\alpha^{-1}+1) \log(\eps^{-1})$)} \, .
\end{align*}}
This completes the proof. 
\end{proof}

\begin{proof}[Proof of Theorem~\ref{theorem:convergence-of-the-dirichlet-problem} under Hypothesis~\ref{item:cont-neighborhood}]
	The assertions corresponding to the hypotheses~\ref{item:planarity} and~\ref{item:diam} follow from Theorems~\ref{theorem:2d} and~\ref{theorem:generaldim-decay} respectively. The assertion corresponding to Hypothesis~\ref{item:vol} follows by applying Theorem~\ref{theorem:generaldim} for each~$n$
	with~$K := 2 \lfloor \frac{\log(1/\min_{v \in \V_n} \vol_{d}(P_v)) }{\log(1/\eps_n)} \rfloor$. Indeed, with this choice, for large enough~$n$ the set~$\mathcal{A}$ defined in~\eqref{eq:small-volume-assumption} is empty but the first term on the right in~\eqref{eq:full-bound} still converges to zero. 
\end{proof}

\section{Convergence of random walk modulo time paramaterization} 
\label{sec:rw-convergence}
In this section we prove a ``black box theorem'' which roughly states that if you have convergence of the Dirichlet 
problem, then the trace of random walk converges, under rescaling, to the trace of Brownian motion. Slightly more precisely,  we show that under this condition, the law of the rescaled random walk converges to the law of Brownian motion with respect to the local topology on curves viewed modulo time parameterization. 

Specifically, we fix a sequence of undirected graphs~$\{\G_n = (\V_n, \E_n)\}_{n\geq0}$
with each~$\V_n \subset \Rd$.  For a point~$z \in \Rd$, let~$z^{n}$ be the nearest vertex in~$\G_n$, with ties broken in lexicographical ordering.  
We also fix a nearest neighbor discrete-time random walk~$\{X_t^{n}\}_{t \geq 0}$ on~$\V_n$; that is,~$\{X_t^{n}\}$ is a time homogeneous Markov chain and~$(X_{t}^{n}, X_{t+1}^{n}) \in \E_n$ for all~$t \geq 0$. We let~$\Delta^{\G_n}$ denote the generator of~$X_t^{n}$, that is,
for every vertex~$a \in \V_n$ and function~$f: \V_n \to \R$, 
\[
\Delta^{\G_n} f(a) = \mathbb{E}[f(X^{a, n}_1) - f(a)] \, , 
\]
where the expected value is taken over one step of the random walk with initial point~$a \in \V_n$.

 We define, for a set~$U \subset \Rd$ and~$r > 0$, 
\[
B_{r}(U) := \{ x \in \Rd : \inf_{y \in U}\| x - y\| \leq r\} \, . 
\]
 The following assumption states that the graph approximates some (possibly infinite) domain~$\mathcal{D}$ and the Dirichlet problem for every ball~$B_r(x)$ intersecting a neighborhood of~$\mathcal{D}$ converges. 
 
\begin{assumptionD}
	\label{ass:D}
	{There exists a domain~$\mathcal{D}$ and some~$\delta > 0$ such that for every~$r > 0$ and~$x \in B_{\delta}(\mathcal{D})$ the following holds for~$B := B_r(x) \cap B_{\delta}(\mathcal{D})$.} The graph approximates Euclidean space in the ball, 
	\begin{equation}
		\label{eq:local-approximation-of-space}
	\lim_{n \to \infty} \sup_{z \in B} \| z^{n} - z\| = 0 \, . 
	\end{equation}
	Moreover, for every function~$h_C$ which is continuum harmonic in a neighborhood of~$B$, the discrete
	harmonic extension~$h^{n}_D$ of~$h_C$ to~$\V_n[B]$, defined by, 
	\[
	\left\{
	\begin{aligned}
		&  \Delta^{\G_n} h^{n}_D = 0 & \mbox{in} & \ \V_n[B] \,,
		\\
		& h^{n}_D = h_C & \mbox{on} &  \ \partial \V_n[B]    \, ,
	\end{aligned}
	\right.
	\]
	converges uniformly to~$h_C$,
	\begin{equation}
		\label{eq:convergence-of-d-problem}
	\lim_{n \to \infty} \sup_{z \in \V_n[B]} |h^{n}_D(z) - h_C(z)|  = 0 \, . 
	\end{equation}
\end{assumptionD}


Under this assumption, we prove a central limit theorem.

\begin{theorem}
	\label{theorem:general-clt}
	Suppose~$\mathcal{D}$ is such that Assumption~\ref{ass:D} holds and let~$U \subset \mathcal{D}$ be a {Lipschitz}, bounded domain. Let~$z \in U$, and let~$X^{z, n}$ be random walk on~$\G_n$ 
	started at~$z^{n}$, let~$J^{z, n}_U$ denote the first exit from~$\V_n[U]$,
	and let~$Y^{z, n}$ denote the piecewise linear interpolation of~$X^{z, n}$.
	Let~$\mathcal{B}^z$ be a standard Brownian motion started at~$z \in U$ with first exit~$\tau^{z}_U$ from~$U$. 
	The supremum over all~$z \in U$ of the Prokhorov distance between the laws of~$Y^{z, n} \vert_{[0,J^{z, n}_U]}$ and~$\mathcal{B}^{z} \vert_{[0, \tau^{z}_U]}$ with respect to the metric~\eqref{eq:cmp-topology} converges to 0 as~$n \to \infty$. 
\end{theorem}

\begin{proof}[Proof of Theorem~\ref{theorem:convergence-of-rw-trace} assuming Theorem~\ref{theorem:general-clt}]
	It follows from Theorem~\ref{theorem:convergence-of-the-dirichlet-problem} that Assumption~\ref{ass:D}
	is satisfied. This immediately implies Theorem~\ref{theorem:convergence-of-rw-trace}.
\end{proof}

\begin{proof}[Proof of Theorem~\ref{theorem:convergence-of-the-dirichlet-problem} assuming Hypothesis~\ref{item:hypb}]
This is immediate from Theorem~\ref{theorem:general-clt} and the fact that the solution of the Dirichlet problem can be expressed in terms of the random walk. 
\end{proof}

It remains to prove Theorem~\ref{theorem:general-clt}. We first show that Assumption~\ref{ass:D} implies convergence of harmonic measure. We then use the strong Markov property  of random walk together with an approximation argument to deduce Theorem~\ref{theorem:general-clt}.

\subsection{Convergence of harmonic measure}
 Let~$\B$ denote the unit ball in~$\Rd$.
Let~$\hm(x)$ be the harmonic measure on~$\partial \B$, as seen from~$x \in \B$. For~$x \in \B$ and~$z \in \partial \V_n[\B]$, we define~$\hmn(x)[z]$ to be the probability measure which assigns 
mass to each vertex of $\partial\V_n[\B]$ equal to its discrete harmonic measure from $x^n$. We also write~$\dPr$ for the Prokohorov distance between two probability measures.

\begin{prop} \label{prop:convergence-of-harmonic-measure}
	Under Assumption~\ref{ass:D}  we have that 
	\[
	\lim_{n \to \infty} \sup_{x \in \B} \dPr( \hmn(x), \hm(x)) = 0 \, . 
	\]
\end{prop}

\begin{proof}
	Let~$f$ be a smooth, bounded function on~$\partial \B$.	By Whitney extension, we may assume that~$f$ is smooth on~$\Rd$.
	It suffices to show that
	\begin{equation}
		\label{eq:weak-convergence-of-smooth}
	\lim_{n \to \infty}  \sup_{x \in \B} \left| \int_{\Rd} f(z) d \hmn(x)[z] - \int_{\partial \B} f(z) d \hm(x)[z] \right| = 0  \, .   
	\end{equation}
 	Fix~$\delta > 0$ and let 
	~$h_C^{\delta}$ denote the unique function which is harmonic on~$B_{1 + \delta}$ with~$h^{\delta}_C = f$ on~$\partial B_{1 + \delta}$. 
	Let~$h^{n,\delta}_D$ denote the unique solution of 	
	\[
	\left\{
	\begin{aligned}
		&  \Delta^{\G_n} h^{n,\delta}_D = 0 & \mbox{in} & \ \V_n[\B] \,,
		\\
		& h^{n,\delta}_D = h^{\delta}_C & \mbox{on} &  \ \partial \V_n[\B]    \, .
	\end{aligned}
	\right.
	\]
	By~\eqref{eq:convergence-of-d-problem}, for each~$\delta > 0$, 
	\[
	\lim_{n \to \infty } \sup_{x \in \B}|h^{n, \delta}_D(x^{n}) - h^{\delta}_C(x)|  = 0 \, . 
	\]
	Also, by smoothness of~$f$ and the continuous dependence of harmonic functions on their boundary data,
	 the maximum principle, and~\eqref{eq:local-approximation-of-space}, 
	\[
	\lim_{n \to \infty} \lim_{\delta \to 0}  \left(  \left| h^{n, \delta}_D(x^{n}) - \int_{\Rd} f(z) d \hmn(x)[z] \right| + \left| \int_{\partial \B} f(z) d \hm(x)[z] - h_C^{\delta}(x) \right| \right) = 0 \, . 
	\]
	Combining the previous two displays with the triangle inequality verifies~\eqref{eq:weak-convergence-of-smooth}, completing the proof. 
\end{proof}

Since Assumption~\ref{ass:D} is preserved under scaling and translation this also shows convergence of the harmonic measure of rescaled and translated balls.

\subsection{Convergence of random walk}
The argument given here is similar to the proofs of~\cite[Lemma 3.14 and Theorem 3.10]{GMSInvariance}.

\begin{proof}[Proof of Theorem~\ref{theorem:general-clt}]
	The idea is to apply rescaled and translated versions of Proposition~\ref{prop:convergence-of-harmonic-measure} to a sequence of small balls following the path of the random walk.

	Fix~$z \in U$, a small parameter~$\delta > 0$, and cover the domain~$B_{2\delta}(U)$ by balls of radius~$\delta$ centered at \emph{grid points} 
	defined as elements of the set $\delta \Zd \cap B_{2 \delta}(U)$.  
	
	We iteratively define, for~$k \in \N$, a sequence of stopping times~$J^{z, n}_{k}$ and grid points~$w_{\delta, k}^{z, n}$.  Start with~$J^{z, n}_0 = 0$ and choose~$w_{\delta,0}^{z, n}$ so that the point~$z^{n} \in \V_n[B_{\delta}(w_{\delta,0}^{z, n})]$. Then, having defined~$w_{\delta, k-1}^{z, n}$ and~$J^{z, n}_{k-1}$,
	if~$w_{\delta, k-1}^{z, n} \not \in B_{\delta}(U)$, then stop, setting~$K_{\delta}^{z, n} := k-1$. Otherwise, let~$w_{\delta, k}^{z, n}$ be the nearest (ties broken lexicographically)	grid point to~$Y^{z, n}_{J^{z, n}_{k-1}}$. Let~$J^{z, n}_{k}$ be 
	the smallest~$t \geq J^{z, n}_{k-1}$ for which~$X_t^{z, n} \in \partial \V_n[B_{2 \delta}(w_{\delta,k}^{z, n})]$. 
	
	The corresponding Brownian motion stopping times are defined as follows. Initialize~$\tau^z_{\delta, 0} = 0$ 
	and the grid point~$w_{\delta, 0}^{z} = w_{\delta, 0}^{z, n}$. Then, having defined~$w_{\delta, k-1}^{z}$ and~$\tau^{z}_{k-1}$,
	if~$w_{\delta, k-1}^z \not \in B_{\delta}(U)$, then stop, setting~$K_{\delta}^z := k-1$. Otherwise, let~$w_{\delta, k}^z$ be the nearest (ties broken lexicographically)	grid point to~$\mathcal{B}^{z}_{\tau^{z}_{k-1}}$. Let~$\tau^{z}_{k}$ be 
	the smallest~$t \geq \tau^{z}_{k-1}$ for which~$\mathcal{B}_t^{z} \in \partial B_{2 \delta}(w_{\delta, k}^z)$.

	Let~$P_{\delta, k}^{z, n}$ denote the conditional law of~$Y^{z, n}_{J_{\delta,k}^{z, n}}$ given~$w_{\delta,k}^{z,n}$
	and let~$\hm_{\delta, k}^{z}$ be the (continuum) harmonic measure 
	on~$\partial B_{2 \delta}(w_{\delta, k}^{z, n})$ as viewed 
	from~$Y^{z, n}_{J_{\delta, k-1}^{z, n}}$.	By (rescaled and translated versions of) Proposition~\ref{prop:convergence-of-harmonic-measure} and the strong Markov property of random walk, 
	\[
	\lim_{n \to \infty} \sup_{z \in U} \sup_{k \in [0, K_{\delta}^{z, n}] \cap \N} \dPr \left( P_{\delta, k}^{z, n}, \hm_{\delta, k}^{z} \right) = 0 \, . 
	\]
	From this, continuity of harmonic measure, and the fact that~$K_{\delta}^z$ and~$K_{\delta}^{z, n}$ are almost surely finite, we find 
	that the supremum over all~$z \in U$ of the Prokhorov distance between the laws of~$\left( Y^{z, n}_{J_{\delta, k}^{z, n}} \right)_{k \in [0, K_{\delta}^{z, n}] \cap \N}$ and~$\left( \mathcal{B}^z_{\tau_{\delta, k}^{z}} \right)_{k \in [0, K_{\delta}^z] \cap \N}$ tends to 0 as~$n \to \infty$. 
	
	Also, by construction and~\eqref{eq:local-approximation-of-space} we have that 
	\begin{equation}
		\label{eq:continuity-of-approximation}
		\sup_{s, t \in [\tau^z_{\delta, k-1}, \tau^z_{\delta, k}]} \|\mathcal{B}_s^z - \mathcal{B}_t^z \| \leq 4 \delta
		\quad \mbox{and} \quad 
		\lim_{n \to \infty} \sup_{s, t \in [J_{\delta, k-1}, J_{\delta, k}]} \| Y_s^{z, n} - Y_t^{z, n} \ | \leq 4 \delta  \, . 
	\end{equation}
	From this and the preceding paragraph, it follows that we can couple the laws of~$Y^{z, n}$ and~$\mathcal{B}^z$
	in such a way that with probability at least~$1-\delta$, we have~$K_{\delta}^{z, n} = K_{\delta}^z$ and 
	\[
	\sup_{k \in [0, K_{\delta}^z] \cap \N} \left| Y^{z, n}_{J_{\delta, k}^{z, n}} - \mathcal{B}_{\tau_{\delta, k}^z}^z \right| \leq \delta \,. 
	\]
	If this is the case, we re-parameterize~$Y^{z, n} \vert_{[0, J_U^{z, n}]}$ in such a way that~$[J_{\delta, k-1}^{z, n}, J_{\delta, k}^{z, n}]$ is traced in~$\tau_{\delta, k}^z - \tau_{\delta, k-1}^z$ units of time.  {Since~$U$ is Lipschitz, a standard Brownian motion estimate
	 shows that~$|\tau_{\delta,K_{\delta}^z} - \tau_z^U| \leq C \delta$ for some constant~$C(U)$ with probability going to one as~$\delta \to 0$.}
	Consequently, with probability going to one as~$\delta \to 0$, by~\eqref{eq:continuity-of-approximation}, the uniform distance between this re-parameterized version of the curve~$Y^{z, n} \vert_{[0,J^{z, n}_U]}$ and~$B^{z} \vert_{[0, \tau^{z}_U]}$ is at most~$C \delta$. Since~$\delta$ and~$z$ were arbitrary, this completes  the proof. 
\end{proof}

\section{Verifying the hypothesis} \label{sec:verifying-assumptions}
In this section we show that Voronoi tesselations of (possibly very degenerate) point processes satisfy Hypothesis~\ref{item:vol}.
That is, we prove Proposition~\ref{prop:min-max-edge-length}.

First, we show a deterministic property of Voronoi tesselations: if the points are sufficiently spread out, 
then the minimal cell size cannot be large, and if the points are not concentrated, then the maximal cell size cannot be small. In the statement and below, we denote by~$[-r,r]^d$ the~$d$-dimensional cube of radius~$r>0$ centered at the origin. 
\begin{lemma} \label{lemma:geometric-constraint}
	Let~$U\subset \Rd$ be a bounded domain and let~$S \subset \Rd$ be a locally finite set of points
	\begin{enumerate}[label=(\alph*)]
		\item If there exists~$r > 0$ with at most one point in~$x + [-2r,2r]^d$ for all~$x \in r \Zd \cap U$, then
		\begin{equation*}
			\label{eq:lower-bound-on-diameter}
		B_{r}(s) \subset \mathcal{C}_s \quad \forall s \in S \cap U \, . 
		\end{equation*}
	
		\item There exists a constant~$\delta = \delta(d) \in (0,1)$ such that if there exists~$k > 0$ with at least one point in~$y + [-\delta k,\delta k]^d$ for all~$y \in k \Zd \cap B_{2 k}(U)$, then
		\begin{equation*}
			\label{eq:upper-bound-on-diameter}
			 \mathcal{C}_s \subset B_{4 k \sqrt{d}}(s) \quad \forall s \in S \cap U \, . 
		\end{equation*}
	\end{enumerate}
\end{lemma} 
To prove Lemma~\ref{lemma:geometric-constraint}, we use a particular case of the following fact: if a convex polytope is large, then the dual polytope must be small. 
\begin{figure}
	\includegraphics[width=0.5\textwidth]{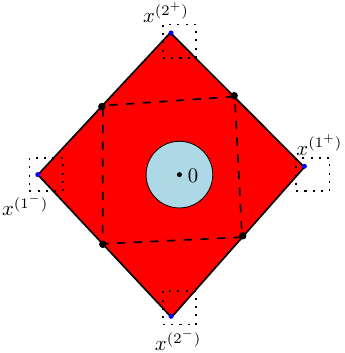}
	\label{fig:nearby-points}
	\caption{Illustration of the proof of Lemma~\ref{lemma:dual-polytope-bound}. The polytope~$P$ is in red and the dual~$P^*$ is outlined by a black dashed line. The cubes~$[-\delta, \delta]^d \pm e_i$ containing~$x^{(i^{\pm})}$, which are blue dots, are outlined by a black dotted line.}
\end{figure}

\begin{lemma} \label{lemma:dual-polytope-bound}
	There exists a constant~$\delta'(d) \in (0,1)$ such that the following holds for every selection of~$x^{(i^{\pm})} \in [-\delta', \delta']^d \pm e_i$ for~$i = 1, \ldots, d$.
	The dual of the convex polytope~$P := \conv(x^{(1^+)}, x^{(1^-)}, \ldots, x^{(d^+)}, x^{(d^-)})$ defined by~$P^* := \{ y \in \Rd: x \cdot y \leq 1 \mbox{ for all~$x \in P$}\}$ is bounded by~$P^* \subset B_{2 \sqrt{d}}$. 
\end{lemma}
\begin{proof}
	It suffices to show that the ball of radius~$R := (4 d)^{-\nicefrac12}$ centered at the origin is contained in the polytope,~$B_R \subset P$. Indeed, 
	by the definition of dual, we have that~$P^* \subset B_R^*$ and since~$B_R^* = B_{R^{-1}}$, this would show~$P^* \subset B_{2 \sqrt{d}}$.

	To see that~$B_R \subset P$, we first observe, by the Cauchy-Schwartz inequality, that the ball of radius~$d^{-\nicefrac12}$ is contained in
	the convex hull of~$ \conv(\pm e_1, \ldots, \pm e_d)$; indeed, if~$z \in B_1$, 
	\[
	\left( \sum_{i=1}^d |z_i| \right)^2 \leq  d \sum_{i=1}^d z_i^2  \leq  d   \, .  
	\]
	Since the map from the set of points to their convex hull 
	is Lipschitz continuous with respect to the Hausdorff distance, 
	this shows that for small enough~$\delta'$, depending only
	on~$d$, the ball of radius~$(4 d)^{-\nicefrac12}$ is contained in~$P$. This completes the proof.
\end{proof}

\begin{proof}[Proof of Lemma~\ref{lemma:geometric-constraint}]
	Proof of (a). Let~$s \in S \cap U$. Let~$s'$ be the nearest, in Euclidean distance, point in~$S$ to~$s$. Let~$x \in r \Zd \cap U$ be such that~$\|s - x\| \leq r$. Then, since there can be no other points in~$x + [-2r, 2r]^d$, we must have~$\|s-s'\| > r$. 
	
	Proof of (b). Let~$s \in S \cap U$ and let~$\delta := c \delta'$ where~$\delta'$ is as in Lemma~\ref{lemma:dual-polytope-bound} and~$c(d) \in (0,1)$ is a constant to be determined below. By rescaling and translation, we may suppose~$s = 0$ and~$k = \frac12$, and that there are, for each~$i \in \{1, \ldots, d\}$, points~$s^{(i^{\pm})} \in \{[- c d \delta', cd   \delta']^d \pm \frac12 e_i\} \cap S$.
	Denote this set of points by~$S'$. Since~$s = 0$, we can rewrite, for all~$s' \in S'$ and~$y \in \Rd$, the halfspace
	  	\[
	  	 (y - \frac{1}{2}(s + s')) \cdot (s-s')  \geq 0 \iff - y \cdot s' + \frac12 \|s'\|^2 \geq 0 \iff y \cdot \left( 2 \frac{s'}{\|s'\|^2} \right) \leq 1 \,  ,
	  	\]
	  	so that by~\eqref{eq:half-space-definition-of-cells} and the fact~$S' \subset S$, we have
	\[
	\mathcal{C}_s \subset \bigcap_{s' \in S'} \left\{ y \in \Rd: y \cdot \left( 2 \frac{s'}{\|s'\|^2} \right) \leq 1  \right\} \, . 
	\]
	Define~$x^{(i^{\pm})} := 2 \frac{s^{(i^{\pm})}}{\|s^{(i^{\pm})}\|^2}$ and observe that for a sufficiently small choice of~$c$, we have $x^{(i^{\pm})} \in [-\delta', \delta']^d \pm e_i$ for each~$i \in \{1, \ldots, d\}$. Thus, Lemma~\ref{lemma:dual-polytope-bound} yields that~$\mathcal{C}_s \subset B_{2\sqrt{d}}$.  This completes the proof after undoing the scaling.
\end{proof}

It remains to show that the conditions of Lemma~\ref{lemma:geometric-constraint} are satisfied with probability one  as~$m \to \infty$
for particular choices of the grid.

\begin{proof}[Proof of Proposition~\ref{prop:min-max-edge-length}]
	We will show that the events in Lemma~\ref{lemma:geometric-constraint} are satisfied for
	\begin{equation}
		\label{eq:choice-of-k-and-r}
	k :=  m^{-1/(2 \beta^{+})} \quad \mbox{and} \quad  r:= m^{-8/\beta^-} \, , 
	\end{equation}
	for all large enough~$m$ for a slightly enlarged set~$U' := B_{4 m^{-1/(4\beta^+)}}(U)$.  For a finite set of points we write~$|\cdot|$ for the cardinality.  We use the fact that the number of points in a cube of side length~$s$, the random variable~$|\Lambda_m \cap [-s,s]^d|$, is Poisson with mean~$m \mu([-s,s]^d)$. 

	For the lower bound, let~$\delta(d)$ be as in Lemma~\ref{lemma:geometric-constraint} so that by~\eqref{eq:holder-assumption}, 
	for all~$m$ sufficiently large, 
	\[
	m \mu([-\delta k,\delta k]^d) \geq  m^{\nicefrac13} \, , 
	\]
	and so
	\[
	\P[|\Lambda_m \cap [-\delta k,\delta k]| = 0] = \exp(-m \mu([-\delta k,\delta k]^d)) \leq \exp(-m^{\nicefrac13})
	\]
	which shows, after taking a union bound,
	\begin{equation}
		\label{eq:lower-bound-cube}
		\P\left[ \bigcup_{y \in k \Zd \cap B_{2 k}(U')} \left\{ \mbox{the cube~$\{y +[-\delta k,\delta k]^d\}$ contains no points of~$\Lambda_m$}\right\} \right] \leq \exp(-m^{\nicefrac14}) \, .
	\end{equation}

	For the upper bound, we first observe that if~$X \sim \mbox{Poisson}(\lambda)$, then, for all~$\lambda$ sufficiently small,
	\begin{equation}
		\label{eq:poisson-bound}
	\P[X \geq 2] = 1 - \P[X=0]  - \P[X=1]  = 1 - \exp(-\lambda) - \lambda \exp(-\lambda) \leq \lambda^2 \, . 
	\end{equation}
	Fix~$N \geq 1$ to be selected below, and let
	\[
	\beta^- = b_0 < \cdots < b_N = \beta^+
	\]
	be a partition of the interval~$[\beta^-, \beta^+]$. For each~$j \in \{1, \ldots, N\}$, let~$R_j$
	be the set of grid boxes of radius~$2 r$ with~$\mu$-volume between~$(2r)^{b_{j}}$ and~$(2r)^{b_{j-1}}$, 
	\[
	R_j := \{ B := y + [-2 r, 2 r]^d  : y \in r \Z^d \cap U' \, \, \mbox{and} \, \, \mu(B) \in [(2 r)^{b_{j}}, (2 r)^{b_{j-1}}]  \} \, . 
	\]
	Observe that by~\eqref{eq:holder-assumption}, for every~$j \in \{1, \ldots, N\}$ we have 
	\begin{equation}
		\label{eq:bound-on-cardinality-of-rj}
		|R_j| \leq r^{-b_{j}} d^d \mu(U') \, ,
	\end{equation}
	and also, 
	\begin{equation*}
		|\Lambda_m \cap B| \sim \mbox{Poisson}(\lambda_j) \quad \forall B \in R_j \quad \mbox{with~$\lambda_j \leq m 2^{\beta^+} r^{b_{j-1}}$} \, .
	\end{equation*}
	 By the previous display and~\eqref{eq:poisson-bound} we have, for all~$j \in \{1, \ldots, N\}$, 
	\[
	\P[	|\Lambda_m \cap B| \geq 2] \leq m^2 4^{\beta^+} r^{2 b_{j-1}} \qquad \forall B \in R_j 
	\]
	and so by a union bound and~\eqref{eq:bound-on-cardinality-of-rj}, 
	\[
		\P\left[ \bigcup_{B \in R_j} \left\{\mbox{the cube~$B$ contains at least two points of~$\Lambda_m$}\right\} \right] \leq  m^2 4^{\beta^+} r^{2 b_{j-1} - b_{j}} \mu(U') \, . 
	\]
	By taking~$N$ sufficiently large, depending only on~$\beta^{+}$ and~$\beta^-$ so that~$b_{j} - b_{j-1} \leq \beta^-/2$, we have~$m^2 r^{2 b_{j-1} - b_j} \leq m^2 r^{\beta^-/2}$. Consequently, by the previous display and a union bound over~$j$, we have 
	\begin{align*}
		&\P\left[ \bigcup_{y \in r \Zd \cap U'} \left\{\mbox{the cube~$\{y +[-2 r, 2r]^d\}$ contains at least two points of~$\Lambda_m$} \right\} \right] \\
		&\qquad \leq  N d^d 4^{\beta^+} m^2 r^{\beta^-/2} \mu(U') \leq N d^d 4^{\beta^+} m^{-2} \mu(U') \, ,
	\end{align*}
	and so for~$m$ sufficiently large
	\begin{equation}
	\label{eq:upper-bound-cube}
			\P\left[ \bigcup_{y \in r \Zd \cap U'} \left\{\mbox{the cube~$\{y +[-2 r,2 r]^d\}$ contains at least two points of~$\Lambda_m$} \right\} \right]  \leq m^{-\nicefrac32}  \, . 
	\end{equation}
	Combining~\eqref{eq:upper-bound-cube} and~\eqref{eq:lower-bound-cube} together with the Borel-Cantelli lemma
	shows that the hypotheses required for Lemma~\ref{lemma:geometric-constraint} are satisfied almost surely 
	for all~$m$ sufficiently large for~$U'$. Since every such cell which intersects~$U$ must be centered at a grid point in~$U'$,  this completes the proof of~\eqref{eq:cells-are-contained-in-balls}.

	This also shows that~$\eps_m := \sup_{v \in \VVor(\Lambda_m)} \diam(P_v) \leq m^{-1/(3\beta^+)}$ for all large~$m$ and that~\eqref{eq:smallness} is satisfied almost surely.
	To see that Hypothesis~\ref{item:vol} is satisfied, we observe that~\eqref{eq:cells-are-contained-in-balls} implies that
	the minimum volume of a cell in~$\G_m$ is~$m^{-d 8/\beta^-} \geq \eps_m^{d (3 \beta^+) 8/\beta^- } $. 
\end{proof}

\section{Counterexample for simple random walk} \label{sec:counter-example}
It is straightforward to construct, in all dimensions, sphere packings of \emph{unbounded} degree upon which simple random walk with unit conductance does not converge to Brownian motion modulo time change, see, for example,~\cite[Example 4.8]{NachmiasStFlour}. 
In this section we give a \emph{bounded} degree example in dimensions at least three. This is similar to the potential examples discussed in~\cite[Section 1.5]{GMSInvariance}.
We are not sure whether such a counterexample exists in two dimensions.

\begin{figure}
	\includegraphics[width=0.45\textwidth]{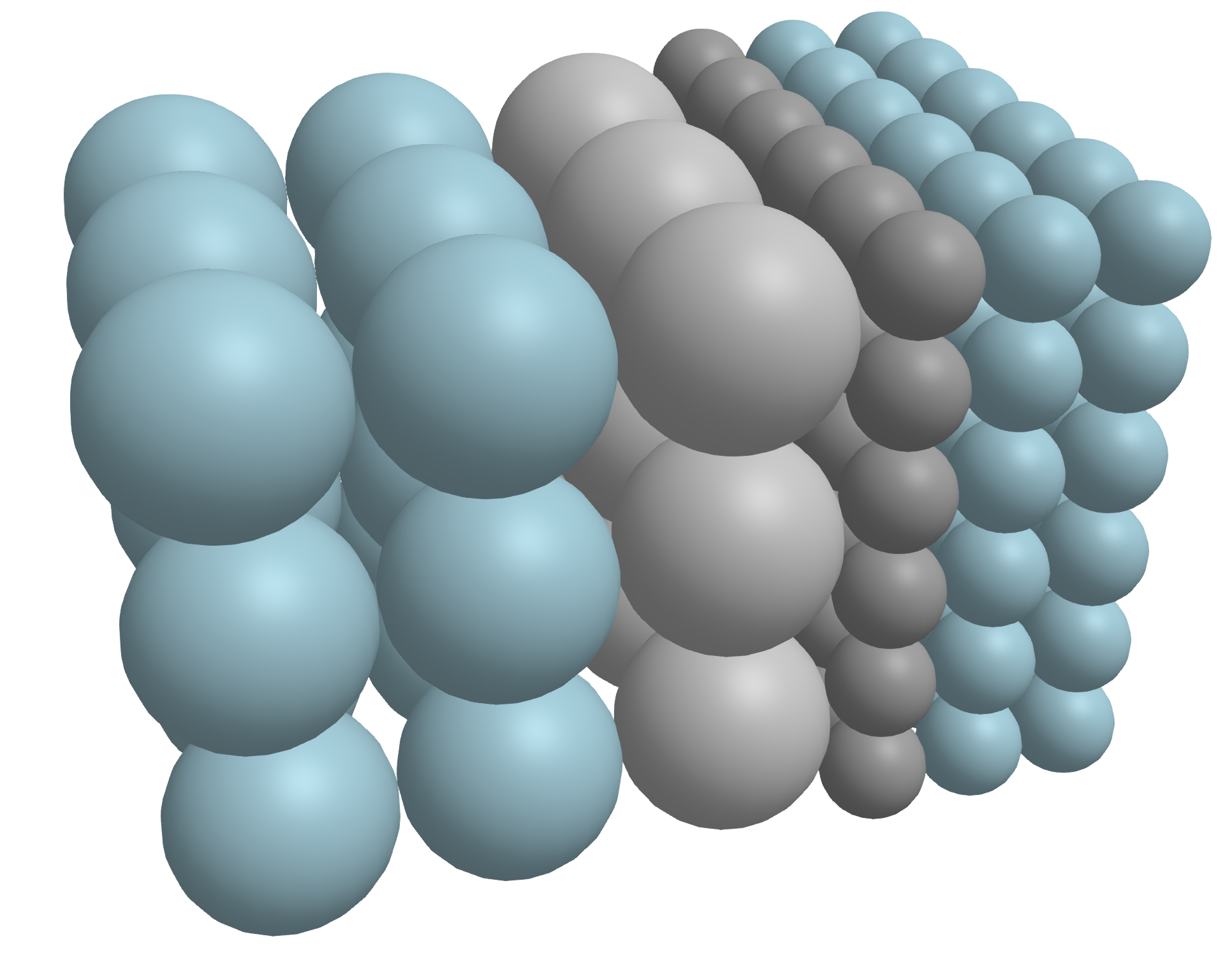}
	\quad 
	\includegraphics[width=0.35\textwidth]{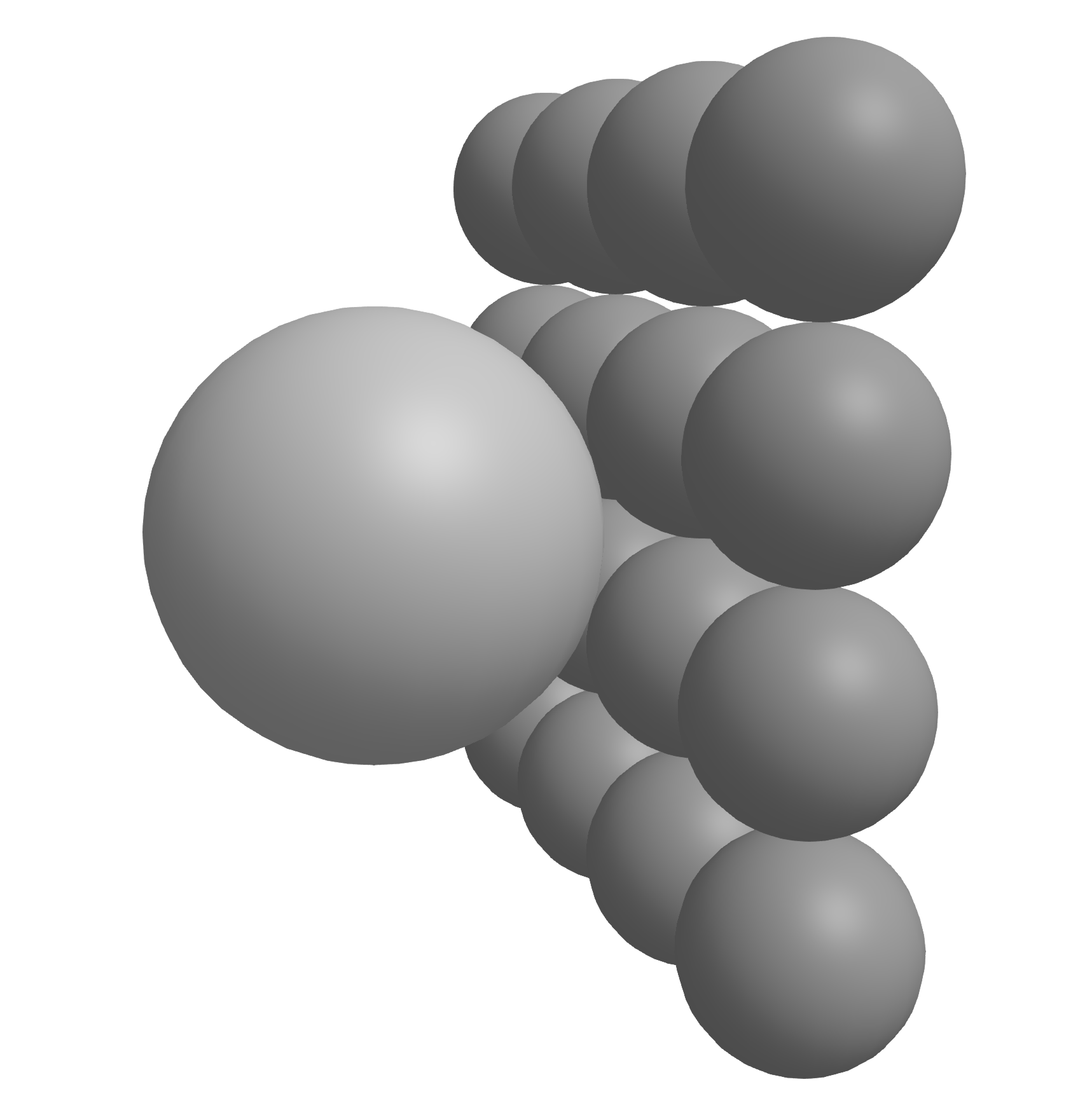}
	\caption{On the left, part of the sphere packing given in Theorem~\ref{theorem:counter-example} for~$d=3$. 
		On the right, a big middle sphere and some nearby middle small spheres.
		Big  (resp.\ small) middles spheres are in light (resp.\ dark) gray and the other spheres are in light blue.}
	\label{fig:counter-example}
\end{figure}

\begin{theorem} \label{theorem:counter-example}
	For all~$d \geq 3$, there exists a sphere packing of~$\R^d$ upon which simple random 
	walk with unit conductance does not converge to Brownian motion modulo time parametrization in the scaling limit. Moreover, the spheres all have radii either $1$ or $1/2$ and are tangent to a bounded number of other spheres.
\end{theorem}

The sphere packing we construct will force the random walk to lack negation symmetry at large scales, precluding it from converging to Brownian motion.

\begin{figure}
	\includegraphics[width=0.5\textwidth]{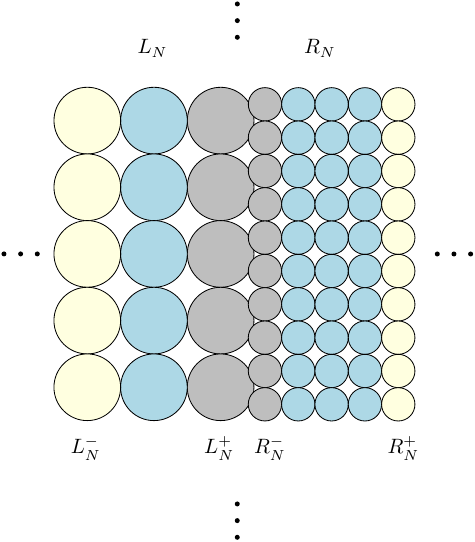}
	\caption{A top-down view of the sphere packing given in Theorem~\ref{theorem:counter-example} for~$d=3$.
		The spheres in~$L_N, R_N$ (defined in~\eqref{eq:def-of-ln}) are in light blue, the spheres in~$L_N^{-}, R_N^{+}$ are in light yellow, 
		the spheres in~$L_N^{+}, R_N^{-}$ are in gray (defined in~\eqref{eq:def-of-lnp}).  Here~$N=2$. Note that the small middle spheres are slightly above small big spheres so they are not overlapping.}
	\label{fig:counter-example-topdown}
\end{figure}

\begin{proof}[Proof of Theorem~\ref{theorem:counter-example}]
	We start with the construction of the sphere packing, part of which is illustrated in Figure~\ref{fig:counter-example}. 
	\begin{itemize}
		\item 	At each vertex~$x$ on~$2\Z^d$ for~$x_1 \leq 0$, place a \emph{big} sphere of radius~$1$. 
		\item 	The big spheres centered at~$x \in \Zd$ for~$x_1 = 0$ are the big~\emph{middle} spheres. 
		\item   Place \emph{small} spheres of radius~$\frac12$ at~$ (\frac{\sqrt{7}}{2}, \frac12, \frac12, 0, \ldots, 0) + y$ for~$y \in \Zd$ with~$y_1 \geq 0$.
		\item The small spheres with centers~$x$ for~$x_1 = \frac{\sqrt{7}}{2}$ are~\emph{small middle} spheres. 
	\end{itemize}

	Two spheres are adjacent to each other if and only if the distance between their centers 
	is equal to the sum of the two radii. Thus, non-middle spheres are adjacent to~$2 d$ other spheres of the same size.  
	Also, by construction, each big middle sphere is adjacent to~$2d-1$ big spheres and~$4$ small middle spheres
	and each small middle sphere is adjacent to~$2d-1$ small spheres and 1 big middle sphere, see Figure~\ref{fig:counter-example}.  Indeed, the big middle sphere with center~$(0, x_2, \ldots, x_d)$
	is adjacent to the four small middle spheres with centers~$(\frac{\sqrt{7}}{2}, x_2\pm \frac12, x_3\pm \frac12, x_4, \ldots, x_d)$ as
	\[
	\sqrt{ \left(\frac{\sqrt{7}}{2}\right)^2 + \left(\pm \frac12\right)^2 + \left(\pm \frac12\right)^2} = 3/2 = 1 + 1/2 \, . 
	\]
	Denote the sphere packing by~$\mathcal{S}$ and identify spheres in~$\mathcal{S}$ by their centers. 
	
	We now argue that simple random walk (SRW) on the tangency graph of~$\mathcal{S}$ is more likely to go right than to go left when started in the middle. To that end, we consider the~$x_1$ coordinate of the walk. Fix~$N \geq 2$ and divide space
	according to Figure~\ref{fig:counter-example-topdown} as follows:
	\begin{equation}
		\label{eq:def-of-ln}
		L_N = \{ x \in \mathcal{S} : -2 N < x_1  < 0 \} \quad \mbox{and} \quad R_N =  \{ x \in \mathcal{S}: \frac{\sqrt{7}}{2} < x_1 < 2N + \frac{\sqrt{7}}{2} \} \,  , 
	\end{equation}
	and consider their boundaries 
	\begin{equation}
		\label{eq:def-of-lnp}
		\begin{aligned}
			L_N^{-} &= \{ x \in \mathcal{S}: x_1 = -2N \} \qquad L_N^{+} = \{ x \in \mathcal{S}: x_1 = 0 \} \\
			R_N^{-} &= \{ x \in \mathcal{S}: x_1 = \frac{\sqrt{7}}{2}  \} \qquad R_N^{+} = \{ x \in \mathcal{S}: x_1 = 2N + \frac{\sqrt{7}}{2} \} \, .
		\end{aligned}
	\end{equation}
	A standard estimate for simple random walk on $\Zd$ (obtained by considering the harmonic function~$x \mapsto x_1$) gives 
	\[
	\P\left[ \mbox{SRW started at~$\{x_1 = -2\}$ hits~$L_N^{-}$ before~$L_N^+$}\right] = \frac{1}{N}
	\]
	and
	\[
	\P\left[ \mbox{SRW started at~$R_N^{-}$ hits~$L_N^{+}$ before~$R_N^+$}\right] = \frac{2N}{2N+1} \, . 
	\]
	Putting these estimates together and denoting 
	\[
	p := \P[\mbox{SRW started at~$L_N^+$ hits~$L_N^{-}$ before~$R_N^+$}]
	\]
	we have, by conditioning on the first exit from~$L_N^+$, using the fact that each big middle sphere is adjacent to 4 small middle spheres,
	\begin{align*}
		p &=  \P\left[ \mbox{SRW started at~$\{x_1 = -2\}$ hits~$L_N^{-}$ before~$L_N^+$}\right] \frac{1}{5}  \\
		&\quad+ p \P\left[ \mbox{SRW started at~$\{x_1 = -2\}$ hits~$L_N^{+}$ before~$L_N^{-}$}\right] \frac{1}{5} \\
		&\quad+  p \P\left[ \mbox{SRW started at~$R_N^{-}$ hits~$L_N^{+}$ before~$R_N^+$}\right] \frac{4}{5} \\
		&=  \frac{1}{5}\left( \frac{1}{N} + p \frac{N-1}{N}\right) + \frac{4}{5} p \frac{2N}{2 N +1}  \, . 
	\end{align*}
	Solving the above equation yields that 
	\[
	p = \frac{2N+ 1}{6N + 1} \, , 
	\]
	in particular, 
	\[
	\lim_{N \to \infty} \P[\mbox{SRW started at~$(0,\ldots,0)$ hits~$R_N^{+}$ before~$L_N^{-}$}]  = \frac{2}{3} \, .  
	\]
	Since Brownian motion satisfies negation symmetry, this is incompatible with the SRW converging to Brownian motion, modulo time parameterization.  
\end{proof}

{
  \bibliographystyle{plainnat}
  \bibliography{refs}

\begin{thebibliography}{41}
\providecommand{\natexlab}[1]{#1}
\providecommand{\url}[1]{\texttt{#1}}
\expandafter\ifx\csname urlstyle\endcsname\relax
  \providecommand{\doi}[1]{doi: #1}\else
  \providecommand{\doi}{doi: \begingroup \urlstyle{rm}\Url}\fi

\bibitem[Aizenman and Burchard(1999)]{AB1999}
M.~Aizenman and A.~Burchard.
\newblock H\"{o}lder regularity and dimension bounds for random curves.
\newblock \emph{Duke Math. J.}, 99\penalty0 (3):\penalty0 419--453, 1999.
\newblock ISSN 0012-7094,1547-7398.
\newblock \doi{10.1215/S0012-7094-99-09914-3}.
\newblock URL \url{https://doi.org/10.1215/S0012-7094-99-09914-3}.

\bibitem[Armstrong et~al.(2019)Armstrong, Kuusi, and Mourrat]{AKMBook}
Scott Armstrong, Tuomo Kuusi, and Jean-Christophe Mourrat.
\newblock \emph{Quantitative stochastic homogenization and large-scale
  regularity}, volume 352 of \emph{Grundlehren der mathematischen
  Wissenschaften [Fundamental Principles of Mathematical Sciences]}.
\newblock Springer, Cham, 2019.
\newblock ISBN 978-3-030-15544-5; 978-3-030-15545-2; 978-3-030-15547-6.
\newblock \doi{10.1007/978-3-030-15545-2}.
\newblock URL \url{https://doi.org/10.1007/978-3-030-15545-2}.

\bibitem[Armstrong and Smart(2016)]{armstrongsmart}
Scott~N. Armstrong and Charles~K. Smart.
\newblock Quantitative stochastic homogenization of convex integral
  functionals.
\newblock \emph{Ann. Sci. \'{E}c. Norm. Sup\'{e}r. (4)}, 49\penalty0
  (2):\penalty0 423--481, 2016.
\newblock ISSN 0012-9593,1873-2151.
\newblock \doi{10.24033/asens.2287}.
\newblock URL \url{https://doi.org/10.24033/asens.2287}.

\bibitem[{B. Cerclé}(2022)]{cercle2022}
{B. Cerclé}.
\newblock {Liouville conformal field theory on even-dimensional spheres}.
\newblock \emph{{Journal of Mathematical Physics}}, 63\penalty0 (1):\penalty0
  013505, 2022.
\newblock URL \url{{https://pubs.aip.org/aip/jmp/article/63/1/013505/2859762}}.

\bibitem[Benjamini and Curien(2011)]{curienbenjamini}
Itai Benjamini and Nicolas Curien.
\newblock On limits of graphs sphere packed in {E}uclidean space and
  applications.
\newblock \emph{European J. Combin.}, 32\penalty0 (7):\penalty0 975--984, 2011.
\newblock ISSN 0195-6698,1095-9971.
\newblock \doi{10.1016/j.ejc.2011.03.016}.
\newblock URL \url{https://doi.org/10.1016/j.ejc.2011.03.016}.

\bibitem[Berestycki and Gwynne(2022)]{GB-LBM}
Nathana\"{e}l Berestycki and Ewain Gwynne.
\newblock Random walks on mated-{CRT} planar maps and {L}iouville {B}rownian
  motion.
\newblock \emph{Comm. Math. Phys.}, 395\penalty0 (2):\penalty0 773--857, 2022.
\newblock ISSN 0010-3616,1432-0916.
\newblock \doi{10.1007/s00220-022-04482-y}.
\newblock URL \url{https://doi.org/10.1007/s00220-022-04482-y}.

\bibitem[Berestycki and Powell(2024)]{berestycki2024gaussian}
Nathana{\"e}l Berestycki and Ellen Powell.
\newblock Gaussian free field and liouville quantum gravity.
\newblock \emph{arXiv preprint arXiv:2404.16642}, 2024.

\bibitem[Boutillier and de~Tili\`ere(2012)]{MR3372861}
C\'{e}dric Boutillier and B\'{e}atrice de~Tili\`ere.
\newblock Statistical mechanics on isoradial graphs.
\newblock In \emph{Probability in complex physical systems}, volume~11 of
  \emph{Springer Proc. Math.}, pages 491--512. Springer, Heidelberg, 2012.
\newblock ISBN 978-3-642-23811-6; 978-3-642-23810-9.
\newblock \doi{10.1007/978-3-642-23811-6\_20}.
\newblock URL \url{https://doi.org/10.1007/978-3-642-23811-6_20}.

\bibitem[Boutillier et~al.(2017)Boutillier, de~Tili\`ere, and
  Raschel]{MR3621833}
C\'{e}dric Boutillier, B\'{e}atrice de~Tili\`ere, and Kilian Raschel.
\newblock The {$Z$}-invariant massive {L}aplacian on isoradial graphs.
\newblock \emph{Invent. Math.}, 208\penalty0 (1):\penalty0 109--189, 2017.
\newblock ISSN 0020-9910,1432-1297.
\newblock \doi{10.1007/s00222-016-0687-z}.
\newblock URL \url{https://doi.org/10.1007/s00222-016-0687-z}.

\bibitem[Chelkak and Smirnov(2011)]{chelkaksmirnov}
Dmitry Chelkak and Stanislav Smirnov.
\newblock Discrete complex analysis on isoradial graphs.
\newblock \emph{Adv. Math.}, 228\penalty0 (3):\penalty0 1590--1630, 2011.
\newblock ISSN 0001-8708,1090-2082.
\newblock \doi{10.1016/j.aim.2011.06.025}.
\newblock URL \url{https://doi.org/10.1016/j.aim.2011.06.025}.

\bibitem[Chelkak and Smirnov(2012)]{MR2957303}
Dmitry Chelkak and Stanislav Smirnov.
\newblock Universality in the 2{D} {I}sing model and conformal invariance of
  fermionic observables.
\newblock \emph{Invent. Math.}, 189\penalty0 (3):\penalty0 515--580, 2012.
\newblock ISSN 0020-9910,1432-1297.
\newblock \doi{10.1007/s00222-011-0371-2}.
\newblock URL \url{https://doi.org/10.1007/s00222-011-0371-2}.

\bibitem[Chelkak et~al.(2023)Chelkak, Laslier, and Russkikh]{CLM}
Dmitry Chelkak, Beno\^{i}t Laslier, and Marianna Russkikh.
\newblock Dimer model and holomorphic functions on t-embeddings of planar
  graphs.
\newblock \emph{Proc. Lond. Math. Soc. (3)}, 126\penalty0 (5):\penalty0
  1656--1739, 2023.
\newblock ISSN 0024-6115,1460-244X.
\newblock \doi{10.1112/plms.12516}.
\newblock URL \url{https://doi.org/10.1112/plms.12516}.

\bibitem[Coudi\`ere et~al.(2001)Coudi\`ere, Gallou\"{e}t, and
  Herbin]{MR1863279}
Yves Coudi\`ere, Thierry Gallou\"{e}t, and Rapha\`ele Herbin.
\newblock Discrete {S}obolev inequalities and {$L^p$} error estimates for
  finite volume solutions of convection diffusion equations.
\newblock \emph{M2AN Math. Model. Numer. Anal.}, 35\penalty0 (4):\penalty0
  767--778, 2001.
\newblock ISSN 0764-583X,1290-3841.
\newblock \doi{10.1051/m2an:2001135}.
\newblock URL \url{https://doi.org/10.1051/m2an:2001135}.

\bibitem[de~Tili\`ere(2007)]{MR3252428}
B\'{e}atrice de~Tili\`ere.
\newblock Scaling limit of isoradial dimer models and the case of triangular
  quadri-tilings.
\newblock \emph{Ann. Inst. H. Poincar\'{e} Probab. Statist.}, 43\penalty0
  (6):\penalty0 729--750, 2007.
\newblock ISSN 0246-0203.
\newblock \doi{10.1016/j.anihpb.2006.10.002}.
\newblock URL \url{https://doi.org/10.1016/j.anihpb.2006.10.002}.

\bibitem[Delmotte(1997)]{delmotte}
T.~Delmotte.
\newblock In\'{e}galit\'{e} de {H}arnack elliptique sur les graphes.
\newblock \emph{Colloq. Math.}, 72\penalty0 (1):\penalty0 19--37, 1997.
\newblock ISSN 0010-1354,1730-6302.
\newblock \doi{10.4064/cm-72-1-19-37}.
\newblock URL \url{https://doi.org/10.4064/cm-72-1-19-37}.

\bibitem[Ding et~al.(2023)Ding, Gwynne, and Zhuang]{ding2023tightness}
Jian Ding, Ewain Gwynne, and Zijie Zhuang.
\newblock Tightness of exponential metrics for log-correlated gaussian fields
  in arbitrary dimension.
\newblock \emph{arXiv preprint arXiv:2310.03996}, 2023.

\bibitem[Droniou et~al.(2018)Droniou, Eymard, Gallou\"{e}t, Guichard, and
  Herbin]{droniou2018gradient}
J\'{e}r\^{o}me Droniou, Robert Eymard, Thierry Gallou\"{e}t, Cindy Guichard,
  and Rapha\`ele Herbin.
\newblock \emph{The gradient discretisation method}, volume~82 of
  \emph{Math\'{e}matiques \& Applications (Berlin) [Mathematics \&
  Applications]}.
\newblock Springer, Cham, 2018.
\newblock ISBN 978-3-319-79042-8; 978-3-319-79041-1.
\newblock \doi{10.1007/978-3-319-79042-8}.
\newblock URL \url{https://doi.org/10.1007/978-3-319-79042-8}.

\bibitem[Dubejko(1999)]{MR1692623}
T.~Dubejko.
\newblock Discrete solutions of {D}irichlet problems, finite volumes, and
  circle packings.
\newblock \emph{Discrete Comput. Geom.}, 22\penalty0 (1):\penalty0 19--39,
  1999.
\newblock ISSN 0179-5376,1432-0444.
\newblock \doi{10.1007/PL00009447}.
\newblock URL \url{https://doi.org/10.1007/PL00009447}.

\bibitem[Duffin(1968)]{duffin}
R.~J. Duffin.
\newblock Potential theory on a rhombic lattice.
\newblock \emph{J. Combinatorial Theory}, 5:\penalty0 258--272, 1968.
\newblock ISSN 0021-9800.

\bibitem[Duminil-Copin et~al.(2018)Duminil-Copin, Li, and Manolescu]{MR3858924}
Hugo Duminil-Copin, Jhih-Huang Li, and Ioan Manolescu.
\newblock Universality for the random-cluster model on isoradial graphs.
\newblock \emph{Electron. J. Probab.}, 23:\penalty0 Paper No. 96, 70, 2018.
\newblock ISSN 1083-6489.
\newblock \doi{10.1214/18-EJP223}.
\newblock URL \url{https://doi.org/10.1214/18-EJP223}.

\bibitem[Evans(2010)]{evans2022partial}
Lawrence~C. Evans.
\newblock \emph{Partial differential equations}, volume~19 of \emph{Graduate
  Studies in Mathematics}.
\newblock American Mathematical Society, Providence, RI, second edition, 2010.
\newblock ISBN 978-0-8218-4974-3.
\newblock \doi{10.1090/gsm/019}.
\newblock URL \url{https://doi.org/10.1090/gsm/019}.

\bibitem[Eymard et~al.(2000)Eymard, Gallou\"{e}t, and Herbin]{eymard2000finite}
Robert Eymard, Thierry Gallou\"{e}t, and Rapha\`ele Herbin.
\newblock Finite volume methods.
\newblock In \emph{Handbook of numerical analysis, {V}ol. {VII}}, volume VII of
  \emph{Handb. Numer. Anal.}, pages 713--1020. North-Holland, Amsterdam, 2000.
\newblock ISBN 0-444-50350-1.
\newblock \doi{10.1016/S1570-8659(00)07005-8}.
\newblock URL \url{https://doi.org/10.1016/S1570-8659(00)07005-8}.

\bibitem[Grimmett and Manolescu(2014)]{MR3201923}
Geoffrey~R. Grimmett and Ioan Manolescu.
\newblock Bond percolation on isoradial graphs: criticality and universality.
\newblock \emph{Probab. Theory Related Fields}, 159\penalty0 (1-2):\penalty0
  273--327, 2014.
\newblock ISSN 0178-8051,1432-2064.
\newblock \doi{10.1007/s00440-013-0507-y}.
\newblock URL \url{https://doi.org/10.1007/s00440-013-0507-y}.

\bibitem[Gurel-Gurevich and Seidel(2022)]{MR4425348}
Ori Gurel-Gurevich and Matan Seidel.
\newblock Recurrence of a weighted random walk on a circle packing with
  parabolic carrier.
\newblock \emph{Israel J. Math.}, 247\penalty0 (2):\penalty0 547--591, 2022.
\newblock ISSN 0021-2172,1565-8511.
\newblock \doi{10.1007/s11856-022-2286-6}.
\newblock URL \url{https://doi.org/10.1007/s11856-022-2286-6}.

\bibitem[Gurel-Gurevich et~al.(2020)Gurel-Gurevich, Jerison, and
  Nachmias]{gurel2020dirichlet}
Ori Gurel-Gurevich, Daniel~C. Jerison, and Asaf Nachmias.
\newblock The {D}irichlet problem for orthodiagonal maps.
\newblock \emph{Adv. Math.}, 374:\penalty0 107379, 53, 2020.
\newblock ISSN 0001-8708,1090-2082.
\newblock \doi{10.1016/j.aim.2020.107379}.
\newblock URL \url{https://doi.org/10.1016/j.aim.2020.107379}.

\bibitem[Gwynne(2020)]{GwynneSurvey}
Ewain Gwynne.
\newblock Random surfaces and {L}iouville quantum gravity.
\newblock \emph{Notices Amer. Math. Soc.}, 67\penalty0 (4):\penalty0 484--491,
  2020.
\newblock ISSN 0002-9920,1088-9477.
\newblock \doi{10.1090/noti}.
\newblock URL \url{https://doi.org/10.1090/noti}.

\bibitem[Gwynne et~al.(2021)Gwynne, Miller, and Sheffield]{GMS-Tutte}
Ewain Gwynne, Jason Miller, and Scott Sheffield.
\newblock The {T}utte embedding of the mated-{CRT} map converges to {L}iouville
  quantum gravity.
\newblock \emph{Ann. Probab.}, 49\penalty0 (4):\penalty0 1677--1717, 2021.
\newblock ISSN 0091-1798,2168-894X.
\newblock \doi{10.1214/20-aop1487}.
\newblock URL \url{https://doi.org/10.1214/20-aop1487}.

\bibitem[Gwynne et~al.(2022)Gwynne, Miller, and Sheffield]{GMSInvariance}
Ewain Gwynne, Jason Miller, and Scott Sheffield.
\newblock An invariance principle for ergodic scale-free random environments.
\newblock \emph{Acta Math.}, 228\penalty0 (2):\penalty0 303--384, 2022.
\newblock ISSN 0001-5962,1871-2509.

\bibitem[Kahane(1985)]{kahane1985chaos}
Jean-Pierre Kahane.
\newblock Sur le chaos multiplicatif.
\newblock \emph{Ann. Sci. Math. Qu\'{e}bec}, 9\penalty0 (2):\penalty0 105--150,
  1985.
\newblock ISSN 0707-9109.

\bibitem[Kallenberg(2021)]{MR4226142}
Olav Kallenberg.
\newblock \emph{Foundations of modern probability}, volume~99 of
  \emph{Probability Theory and Stochastic Modelling}.
\newblock Springer, Cham, third edition, 2021.
\newblock ISBN 978-3-030-61871-1; 978-3-030-61870-4.
\newblock \doi{10.1007/978-3-030-61871-1}.
\newblock URL \url{https://doi.org/10.1007/978-3-030-61871-1}.

\bibitem[Kenyon(2002)]{kenyon2002}
R.~Kenyon.
\newblock The {L}aplacian and {D}irac operators on critical planar graphs.
\newblock \emph{Invent. Math.}, 150\penalty0 (2):\penalty0 409--439, 2002.
\newblock ISSN 0020-9910,1432-1297.
\newblock \doi{10.1007/s00222-002-0249-4}.
\newblock URL \url{https://doi.org/10.1007/s00222-002-0249-4}.

\bibitem[Mercat(2001)]{mercat}
Christian Mercat.
\newblock Discrete {R}iemann surfaces and the {I}sing model.
\newblock \emph{Comm. Math. Phys.}, 218\penalty0 (1):\penalty0 177--216, 2001.
\newblock ISSN 0010-3616,1432-0916.
\newblock \doi{10.1007/s002200000348}.
\newblock URL \url{https://doi.org/10.1007/s002200000348}.

\bibitem[M{\o}ller(1994)]{mollerLectures}
Jesper M{\o}ller.
\newblock \emph{Lectures on random {V}orono\u{\i} tessellations}, volume~87 of
  \emph{Lecture Notes in Statistics}.
\newblock Springer-Verlag, New York, 1994.
\newblock ISBN 0-387-94264-5.
\newblock \doi{10.1007/978-1-4612-2652-9}.
\newblock URL \url{https://doi.org/10.1007/978-1-4612-2652-9}.

\bibitem[Nachmias(2020)]{NachmiasStFlour}
Asaf Nachmias.
\newblock \emph{Planar maps, random walks and circle packing}, volume 2243 of
  \emph{Lecture Notes in Mathematics}.
\newblock Springer, Cham, 2020.
\newblock ISBN 978-3-030-27967-7; 978-3-030-27968-4.
\newblock \doi{10.1007/978-3-030-27968-4}.
\newblock URL \url{https://doi.org/10.1007/978-3-030-27968-4}.
\newblock \'{E}cole d'\'{e}t\'{e} de probabilit\'{e}s de Saint-Flour
  XLVIII---2018, \'{E}cole d'\'{E}t\'{e} de Probabilit\'{e}s de Saint-Flour.
  [Saint-Flour Probability Summer School].

\bibitem[Rhodes and Vargas(2014)]{rhodes2014gaussian}
R\'{e}mi Rhodes and Vincent Vargas.
\newblock Gaussian multiplicative chaos and applications: a review.
\newblock \emph{Probab. Surv.}, 11:\penalty0 315--392, 2014.
\newblock \doi{10.1214/13-PS218}.
\newblock URL \url{https://doi.org/10.1214/13-PS218}.

\bibitem[Schiavo et~al.(2021)Schiavo, Herry, Kopfer, and
  Sturm]{schiavo2021conformally}
Lorenzo~Dello Schiavo, Ronan Herry, Eva Kopfer, and Karl-Theodor Sturm.
\newblock Conformally invariant random fields, quantum liouville measures, and
  random paneitz operators on riemannian manifolds of even dimension.
\newblock \emph{arXiv preprint arXiv:2105.13925}, 2021.

\bibitem[Sheffield(2023)]{MR4680280}
Scott Sheffield.
\newblock What is a random surface?
\newblock In \emph{ICM---International Congress of Mathematicians. Vol. 2.
  Plenary lectures}, pages 1202--1258. EMS Press, Berlin, 2023.
\newblock ISBN 978-3-98547-060-0, 978-3-98547-560-5, 978-3-98547-058-7.

\bibitem[Skopenkov(2013)]{skopenkov2013boundary}
M.~Skopenkov.
\newblock The boundary value problem for discrete analytic functions.
\newblock \emph{Adv. Math.}, 240:\penalty0 61--87, 2013.
\newblock ISSN 0001-8708,1090-2082.
\newblock \doi{10.1016/j.aim.2013.03.002}.
\newblock URL \url{https://doi.org/10.1016/j.aim.2013.03.002}.

\bibitem[Smirnov(2010)]{smirnovicm}
Stanislav Smirnov.
\newblock Discrete complex analysis and probability.
\newblock In \emph{Proceedings of the {I}nternational {C}ongress of
  {M}athematicians. {V}olume {I}}, pages 595--621. Hindustan Book Agency, New
  Delhi, 2010.
\newblock ISBN 978-81-85931-08-3; 978-981-4324-31-1; 981-4324-31-0.

\bibitem[Telcs(2006)]{Telcs2006}
Andr\'{a}s Telcs.
\newblock \emph{The art of random walks}, volume 1885 of \emph{Lecture Notes in
  Mathematics}.
\newblock Springer-Verlag, Berlin, 2006.
\newblock ISBN 978-3-540-33027-1; 3-540-33027-5.

\bibitem[Werness(2015)]{werness2015discrete}
Brent~M Werness.
\newblock Discrete analytic functions on non-uniform lattices without global
  geometric control.
\newblock \emph{arXiv preprint arXiv:1511.01209}, 2015.

\end{thebibliography}
}

\end{document}